\documentclass[11pt]{amsart}
\usepackage{amssymb,float}
\usepackage[mathscr]{euscript}
\usepackage{tikz} 
\usetikzlibrary{arrows,positioning,decorations.pathmorphing}
\tikzset{inner sep=0pt, 
  root/.style={circle,draw,minimum size=11pt,thick}, 
  short root/.style={circle,fill,minimum size=7pt}, 
} 
\usepackage{chngcntr}
\counterwithin{figure}{section}
\newcommand{\mbb}{\mathbb}
\newcommand{\mbf}{\mathbf}

\newcommand{\om}{\xi}
\usepackage{graphicx}
\newtheorem{thm}[equation]{Theorem} 
\newtheorem{prop}[equation]{Property}
\newtheorem{proposition}[equation]{Proposition}
\newtheorem{lemma}[equation]{Lemma} 
\newtheorem{lm}[equation]{Lemma} 
\newtheorem{cor}[equation]{Corollary}
\newtheorem{example}[equation]{Example}
\newtheorem{remark}[equation]{Remark}
\newtheorem{definition}[equation]{Definition}

\newcommand{\la}{\langle}
\newcommand{\ra}{\rangle}
\newcommand{\op}{\operatorname}

\newcommand{\ZZ}{\mathbf Z}

\newcommand{\G}{\mathbf G}
\renewcommand{\H}{\mathbf H}
\newcommand{\K}{\mathbf K}
\newcommand{\M}{\mathbf M}
\renewcommand{\S}{\mathbf S}

\newcommand{\V}{\mathbf V}
\newcommand{\W}{\mathbf W}
\newcommand{\m}{\mathbf m}
\newcommand{\n}{\mathbf n}
\renewcommand{\r}{\mathbf r}

\renewcommand{\u}{\mathbf u}
\renewcommand{\v}{\mathbf v}
\newcommand{\w}{\mathbf w}
\newcommand{\x}{\mathbf x}
\newcommand{\y}{\mathbf y}
\newcommand{\z}{\mathbf z}
\newcommand{\CC}{\mathbb C}
\newcommand{\robust}{greed compatible}
\newcommand{\Sym}{\text{Sym}}
\begin{document}

\title[Group coding with complex isometries]
{Group coding with complex isometries}

\author{Hye Jung Kim} 
\address{Dept.\ of Mathematics, University of Hawai`i-West O`ahu, 
Kapolei, HI 96707, USA}
\email{hyejungkimkim@gmail.com}
\author{J.\ B.\ Nation} 
\address{Dept.\ of Mathematics, University of Hawai`i, Honolulu, 
HI 96822, USA}
\email{jb@math.hawaii.edu}
\author{Anne V. Shepler} 
\address{Dept.\ of Mathematics, University of North Texas, Denton, 
TX 76203, USA}
\email{ashepler@unt.edu}

\dedicatory{In memory of Wes Peterson.}
\thanks{
Results were presented at the RIMS Workshop on Combinatorial Structures
and Information Theory in Ashikaga, Japan in August 2010.
The first author was partially supported by NSF research grants
\#DMS-0800951 and \#DMS-1101177.}

\date{November 4, 2013.}

%%%%%%%%%%%%%%%%%%%%%%%%%%%%%%%%%%%%%%%%%%%%%%%%%%%5
\begin{abstract}
We investigate group coding for arbitrary 
finite groups acting linearly on a vector space.  
These yield robust codes based on real or complex
matrix groups.
We give necessary and sufficient
conditions for correct subgroup decoding 
using geometric notions of minimal length coset representatives.
The infinite family of complex reflection groups
$\G(r,1,n)$ 
produces effective codes of arbitrarily large size that can be decoded
in relatively few steps.
\end{abstract}
%%%%%%%%%%%%%%%%%%%%%%%%%%%%%%%%%%%%%%%%%%%%%%%%%5

\maketitle

\section{Introduction}
Permutation group codes originated  
in the 1950's in unpublished memos of David Slepian, 
who used the orbit of a point on a 
sphere under a group action
as signals for communication.
Slepian chose
a group of permutations of coordinates and reversals of 
their signs acting on a finite-dimensional real vector space.
He published this work in 1965 
and extended the idea to arbitrary groups of isometries 
(see~\cite{DS1} and~\cite{DS2}).
Ingemarsson~\cite{II} and Ericson~\cite{TE} provide
surveys of early work on group codes.
Recent applications of permutation codes to flash memory can be found in
Jiang \emph{et al.}~\cite{JMSB,JSB} and Barg and Mazumdar~\cite{BM}.

Slepian's original permutation group codes 
have been generalized to other real reflection groups (Coxeter groups);
see Mittelholzer and Lahtonen~\cite{ML} for a comprehensive account. 
Fossorier, Nation, and Peterson~\cite{FNP} 
developed a decoding method for group codes
using a sequence of subgroups and coset representatives
which yields
efficient decoding of real reflection group codes.
Properties of the length function (defined by simple reflections)
and parabolic subgroup structure give effective codes 
based on Coxeter groups.
Peterson asked what other groups might have an action that lends
itself well to coding using these ideas.  

In this note, we analyze
properties that an arbitrary finite group of 
complex matrices should exhibit for
a successful group coding scheme.  
After outlining group coding and subgroup decoding 
in Section~\ref{description}, we enumerate 
in Section~\ref{effective} the 
characteristics of an effective code.
Section~\ref{geometricnotions}
establishes various geometric notions of minimal coset representatives
analogous to minimal length representatives 
in the theory of Coxeter groups.
These representatives are defined with respect to some
fixed initial vector and sequence of nested subgroups.
We use analogs of Weyl chambers
for arbitrary isometry groups.
We prove that
these geometric notions yield robust codes
in Section~\ref{greedpays} and 
give necessary and sufficient conditions for correct subgroup
decoding in Section~\ref{minimalitysuffices}.  

To summarize two main results from these sections,
let us distinguish two levels of ``correct decoding."  
We say that an algorithm decodes correctly \emph{with some noise}
if there exists $\delta>0$ such that a received vector $\r$ decodes
to a transmitted codeword $\w$ whenever $\| \r - \w \| < \delta$.
We say that the algorithm decodes \emph{robustly}
if a received vector $\r$ always decodes to the nearest codeword $\w$.

\vspace{1ex}

{\bf Theorem A.}
{\em Let\/ $\G$ be any finite matrix group and choose any initial
vector with full orbit and any sequence of nested subgroups.
The subgroup decoding algorithm decodes correctly 
with some noise if and only
if induced coset representatives are minimal.}

\vspace{1ex}

{\bf Theorem B.}
{\em Let\/ $\G$ be any finite matrix group and choose any initial
vector and sequence of nested subgroups.
If coset representatives are greed compatible,
then the group decoding algorithm decodes robustly.
}

\vspace{2ex}

In Section~\ref{comparingnotions}, we compare our
various geometric notions of minimal coset representatives.
We discuss ties in Section~\ref{ties}
and show how to improve the efficiency of decoding in Section~\ref{type3}.
We give a result on controlling error in Section~\ref{errorcontrol}
using group theory.
These ideas are implemented for general wreath
products (of an isometry group with a symmetric group)
in Section~\ref{wreath}.
After a quick background on reflection groups
in Section~\ref{reflections},
we apply our ideas by constructing and analyzing 
effective group codes built on the infinite family of 
complex reflection groups
$\G(r,1,n)$ in Section~\ref{gr1n}.
These codes include previous codes based on the Coxeter groups 
$\text{Sym}_{n}$ (the symmetric groups) and $W\! B_n$
(the hyperoctahedral groups).
The family $\G(r,1,n)$
offers group codes of arbitrarily large size 
with low decoding complexity
that carry special geometric significance:
For each $n,r>1$, the group
$\G(r,1,n)$ is the symmetry group
of a Platonic solid
in $n$-dimensional complex space,
the generalized $r$-cube or ``cross polytope''.
Note that with few exceptions (thirty-four, actually), every irreducible
complex reflection group is some $\G(r,1,n)$ or one of its subgroups.

For some other complex reflection groups,
the subgroup decoding methods
described here do not work as well, 
as we explain in Section~\ref{othergroups}.
It can be unclear how to adjust the parameters so that  
encoded messages decode
correctly.  For these cases,
the first author~\cite{HJK} has developed alternate decoding algorithms
which have been refined by Walker~\cite{CW}
(see also~\cite{JBNCW}).   Appendix~I describes a general
version of this alternate decoding scheme 
and gives a sufficient condition for correct decoding.
Appendix~II outlines a method to improve the performance of codes based
on $\G(r,1,n)$ using a proper subset of the orbit of the initial vector as the set
of codewords.

Of course, there are other encoding/decoding schemes for group codes
which could likely extend well to complex
reflection groups. 
Besides the more traditional
sorts of group decoding schemes using sorting algorithms,
Hagiwara, Kong, and Wadayama (see~\cite{TWMH, MHJK})
have recently introduced permutation codes with linear programming 
decoding.  This seems to be a particularly interesting approach.

Note that any finite group of complex linear transformations
acts by isometries with respect to some inner product.
(One may just average an arbitrary inner product on the
vector space over the finite group to produce
one that is invariant under the group action.)
After a possible change of basis, we may assume this inner
product is standard, and thus the finite group acts by
unitary matrices.  We occasionally
use this assumption when it simplifies arguments.

Also note that we have attempted to make arguments amenable to both
pure mathematicians and coding theorists.

%%%%%%%%%%%%%%%%%%%%%%%%%%%%%%%%%%%%%%%%%%%%%%%%%%%%%
\section{Description of the subgroup decoding scheme}
\label{description}
We distinguish different levels of generality in discussing
group coding schemes, beginning with the basic method before 
proceeding to more detailed algorithms.  
Mathematical readers should recall 
that the goal of coding is not encryption, 
but rather the efficient transmission or storage of information while 
resisting channel noise (corruption) and controlling errors.
There is no explicit error correction involved in group coding; rather, one may superimpose
a correction scheme after the received vector is
decoded.  

%%%%%%%%%%%%%%%%%%%%%%%%%%%%%%%%%%%%%%%%%%%%%%%%%%%%%%%%%%%%%%%%%%
We fix a finite group $\G$ of isometries 
acting on a finite dimensional vector space $\V$.
To simplify notation, we assume $\V$ is a complex
vector space, and so we may assume $\G$ is a unitary group. 
Our arguments extend to isometry groups over other spaces
as well: We could just as well take $\V$
to be a real vector space and $\G$ a group of orthogonal
matrices,  
or take $\V$ to be a vector space over the
division ring $\mbb H$ of real quaternions so that $\G$ consists of unitary 
matrices over $\mbb H$.

%%%%%%%%%%%%%%%%%%%%%
\subsection{Group coding scheme}

A \emph{group coding scheme} 
uses 
the following general method for encoding and decoding, without specifying the details of implementation.
Identify a set of messages $\M$ with group elements
using some correspondence, $\gamma: \M \to \G$.
Fix an \emph{initial vector} $\x_0$ on the unit sphere in $\V$.
(We standardize the initial vector to length one by convention.) 
The \emph{code} is the orbit
of the initial vector under the group $\G$,
$$\G\x_0  = \{ g\x_0 : g \in \G \},$$
and the points $g\x_0$ are called \emph{codewords}.
(More generally, coding theory often uses a subset of the orbit
of $\x_0$ as the code; e.g., see Appendix II.)
We send a message $\mbf m$ in $\M$ to some receiver
by transmitting the corresponding codeword,
$$
\x = g^{-1} \x_0\quad
\text{({\em transmitted vector} or {\em coded message})},
$$
where $g = \gamma(\mbf m)$.  Interference may disrupt
communication,  
and the received vector (which may no longer lie
on the unit sphere) generally 
has the form 
$$
\r = \x + \n\quad
\text{({\em received vector})},$$ 
where $\n$ in $\V$ represents channel noise.  
Ideally, $\r$ will be close to $\x$, i.e., the 
distance $\Vert \r - \x \Vert$ will be small with respect
to the given $\G$-invariant inner product on $\V$.
The receiver decodes by finding a group element $g'$ that 
maps $\r$ as close as possible to the initial vector $\x_0$:
$$
g' \text{\ ({\em decoded message})}
\text{ minimizes }
\Vert a\r - \x_0 \Vert \text{ over all } a \text{ in } \G\, .
$$
The received message is then the message corresponding
to $g'$, i.e., $\mbf m' = \gamma^{-1} (g')$.
We call $g$ the {\em sent message}
and $g'$ the {\em decoded message},
suppressing the dependence on some choice of $\gamma$.

%%%%%%%%%%%%%%%%%%%%%%%%%%%%%%%%%%%%%%%%%%%%%%%%%%%%%%%%%
\subsection{Orbit of the initial vector}
A natural ambiguity arises
as the group coding scheme may not output a unique
decoded message $g'$ for each sent message $g$:
the received vector may be equidistant from two different
codewords.
We say that the initial vector $\x_0$ has {\em full orbit}
if the size of its orbit is the order of the group $\G$.
If $\x_0$ does not have full orbit,
then the isotrophy (point-wise fixer) subgroup
$$\S=\op{Stab}_{\G}(\x_0)$$ of $\x_0$ in $\G$ is nontrivial,
and several group elements $a$ could minimize the distance
between $a\r$ and $\x_0$, since 
$$
||a\r -\x_0|| = ||a'\r- \x_0||
$$
for all $a,a'$ in the same right coset of $\S$ (i.e., with
$\S a=\S a'$).
Thus, we say two group elements define
{\em equivalent} messages if they lie in the same right coset of $\S$.
We seek a decoding method that
outputs messages equivalent to those sent.

Subgroup decoding works better and the theory is more transparent
when $\x_0$ has full orbit,
and one can always choose an initial vector with full orbit.
(If $\G$ is a reflection group, for example, we fix a vector
$\x_0$ off a reflecting hyperplane.)
So why have we chosen to keep track
of $\S$  (see Theorem~\ref{algorithmworks}) 
before emphasizing the case of initial vectors with full orbit?
Some readers may wish to apply the theory of group coding
presented here to arbitrary representations of an 
abstract finite group (which may not act faithfully).
In fact, it is not customary in coding theory 
to always use an initial vector with full orbit, 
and indeed, some interesting codes arise
from other choices (see~\cite{ML, FNP, TWMH}).  
In any case, 
a nontrivial isotrophy subgroup $\S$ is not an obstacle, 
as we may replace $\gamma$ by a map
from messages to representatives of right cosets of $\S$
and define a left inverse map $\gamma^{-1}$ that is constant on right
cosets of $\S$.  

%%%%%%%%%%%%%%%%%%%%%%%%%%%%%%%%%%%%%%%%%%%%%%%%%%%%%%
\subsection{Basic subgroup decoding}

When the group $\G$ is finite but
large, it is not efficient to loop through all the elements 
$a$ in $\G$ to determine those that minimize
$||a\r -\x_0||$ and obtain
the decoded message.  
There are various methods to organize the search,
among which is the 
\emph{basic subgroup decoding algorithm}, which we explain now.

For any nested subgroups $\H < \K$ of $\G$, we may fix  
a set $\op{CL}(\K/ \H)$ of coset representatives
for the left cosets of $\H$ in $\K$
(i.e., the sets $a \H$ for $a$ in $\K$) that includes $I$.
These representatives are called \emph{coset leaders}
of $\K$ over $\H$ following traditional coding theory terminology.

The parameters at our disposal for basic subgroup decoding are
\begin{itemize}
\item a finite group $\G$ of isometries acting on the vector space $\V$,
\item an initial vector $\x_0$ with $\Vert \x_0 \Vert = 1$,
\item a sequence of nested subgroups 
\[ \{ I \} = \G_0 < \G_1 < \G_2 \ldots < \G_m=\G,\text{ and }  \]
\item  coset leaders $\op{CL}(\G_k/\G_{k-1})$ 
 for $\G_k$ over $\G_{k-1}$.
\end{itemize}

Every element of $\G$ has a unique expression as a product of 
coset leaders, giving a ``canonical form'' for group elements:
We may uniquely write any element $g$ in $
\G$
as $g = c_m \cdots c_1$ with each $c_k$ in $\op{CL}(\G_k/\G_{k-1})$.
Thus, the transmitted codeword corresponding to the encoded message 
$g = \gamma(\m)$ can be written as 
\[ \x = g^{-1} \x_0 = c_1^{-1} \cdots c_m^{-1} \x_0 .  \]

The recursive {\em subgroup 
decoding algorithm} is defined as follows.
Let $\r = \x + \n$ denote the received vector and set $\r_0=\r$.
At the $k$-th step, assume
$\r_{k-1} = d_{k-1} \cdots d_1 \r$ is given
for some sequence of coset leaders
$d_j \in \op{CL}(\G_j/\G_{j-1})$.
Find a coset leader $d_k \in \op{CL}(\G_k/\G_{k-1})$ that minimizes 
the distance $\Vert a\r_{k-1} - \x_0 \Vert$ 
over all $a \in \op{CL}(\G_k/ \G_{k-1})$ and set
$\r_k = d_k \r_{k-1} = d_k \cdots d_1 \r$.
(If more than one coset leader yields the minimum distance, 
choose the first one in some
ordering.)  After $m$ steps, the algorithm outputs
\[ g'=d_m\cdots d_1 \]
and the decoded message is interpreted as $\m' = \gamma^{-1}(g')$.

For certain groups $\G$, subgroup sequences, and choices
of initial vector, the element 
$g'$ always minimizes the distance 
$\Vert a\r - \x_0 \Vert$ over
all $a \in \G$ for small noise and is
equivalent to the sent message $g$.  The 
coding scheme then decodes correctly
and resists corruption by noise.

One could test all coset leaders at each step
of the subgroup decoding algorithm to find a minimizing coset
leader, but we explain a more efficient method 
in Section~\ref{type3}.  One
navigates recursively through  
a spanning tree of the coset leader graph, yielding
a {\em standard subgroup decoding algorithm}.
This method has been shown to work well 
for real reflection groups (see~\cite{FNP}) 
and can be very efficient.  

%%%%%%%%%%%%%%%%%%%%%%%%%%%%%%%%%%%%%%%%%%%%%%%%%%%%%%%
%%%%%%%%%%%%%%%%%%%%%%%%%%%%%%%%%%%%%%%%%%%%%%%%%%%%%
%%%%%%%%%%%%%%%%%%%%%%%%%%%%%%%%%%%%%%%%%%%%%%%%%%%%%
\section{Effective Decoding}\label{effective}
What does it mean for a decoding scheme to work effectively?
It should decode correctly
despite channel noise and 
implement practically.
One can ask whether an algorithm
\begin{enumerate}
\item decodes correctly with no noise,
\item decodes correctly with some noise,
\item decodes robustly, i.e., always decodes to the nearest 
codeword,
\item controls error when noise is large, and
\item decodes in a reasonably small number of steps.
\end{enumerate}
We will address these questions in order for the subgroup decoding algorithm.

Correct decoding occurs when
the decoding algorithm
outputs a message equivalent to $g$ 
whenever the code word $g^{-1}(\x_0)$ is transmitted.
In this case, the greedy algorithm
produces a global minimum (of distance back to the initial vector
$\x_0$)
even though, at each stage of the algorithm,
only coset leaders are tested for finding local minimums.
It is not clear that the initial vector and coset leaders
can always be adjusted to ensure correct decoding
after a subgroup sequence has been fixed.
(For example, see the code based on
the exceptional complex reflection group $\G_{25}$ in~\cite{HJK}.)  
This explains why the conditions for decoding
correctly with noise in Sections~\ref{greedpays} and~\ref{minimalitysuffices}  
are somewhat involved.

%%%%%%%%%%%%%%%%%%%%%%%%%%%%%%%%%%%%%%%%%%%%%%%%%%%%%%

Except for artificial examples, however, a decoding scheme that works 
with zero noise will also decode correctly whenever the received vector
is in some neighborhood of a codeword.  This can be formalized:

%%%%%%%%%%%%%%%%%%%%%%%%%%%%%%%%%%%%%
\begin{definition}
We say that an algorithm \emph{decodes correctly with some noise} if there
exists $\delta > 0$ such that
a received vector $\r$ decodes to a group element equivalent to $g$ whenever
$\Vert \r - g^{-1}\x_0 \Vert < \delta$.   
\end{definition}
%%%%%%%%%%%%%%%%%%%%%%%%%%%%%%%%%%%%%%%%%%

Corollary~\ref{correctdecoding}
gives necessary and sufficient conditions for correct decoding 
with some noise.
A stronger notion of correct decoding 
requires received vectors to decode to closest codewords
when they exist:

%%%%%%%%%%%%%%%%%%%%%%%%%%%%%%%%%%%%%%%%%%%%%%%%%%%%%%%%
\begin{definition}
We say that an algorithm \emph{decodes robustly} 
if a received vector $\r$ decodes to a group element
equivalent to $g$ in $\G$ whenever $\r$ is closer to codeword
$g^{-1}\x_0$ than any other codeword, i.e., whenever
$\Vert \r - g^{-1}\x_0 \Vert < \Vert \r - h^{-1}\x_0 \Vert$  
for all $h \notin \S g$.
\end{definition}

Robust decoding is of course desirable and implies correct decoding with some noise.  But it is not always easy to 
verify robust decoding, 
while it is often straightforward to check that an algorithm
decodes correctly with some noise.
A sufficient condition for robust decoding is given in
Theorem~\ref{algorithmworks} and applied in Section~\ref{gr1n}
to the codes based on
the groups $\G(r,1,n)$. 

The fourth property can also be interpreted geometrically
using abstract group theory:
If the received vector $\r$ is closer to a codeword $h^{-1}\x_0$
than to the transmitted vector $g^{-1}\x_0$, 
then the algorithm will output decoded message $h$ instead of
$g$ when decoding correctly (up to equivalence by the isotrophy
subgroup of $\x_0$).  Thus, we may 
control error even with large noise
by choosing the correspondence $\gamma$ between
messages and group elements so that
$\gamma^{-1}(g)$ and $\gamma^{-1}(h)$ do not differ much 
whenever $h^{-1}\x_0$ and $g^{-1}\x_0$ are close,
at least with high probability.
For the purposes of this paper, we take the message $\gamma(g)$ to 
be the actual sequence of coset leaders $c_1, \ldots, c_m$ such that
$g=c_m \cdots c_1$. More generally, $\gamma$ could be some function
of this sequence, e.g., a bitstring determined by the coset leaders.
(Each coset leader could determine a piece of a long 
bitstring, for example.)
Thus, we arrange a subgroup decoding algorithm so that if 
$g=c_m \cdots c_1$ and $h=d_m \cdots d_1$ with 
$\Vert g^{-1}\x_0 - h^{-1}\x_0 \Vert$ sufficiently small, then
$c_i = d_i$ for almost all $i$, thereby controlling error when 
interference produces large noise.
This is the effect of Theorem~\ref{thmWes}.

The fifth property can be analyzed by counting the number of operations
in the algorithm (in some reasonable way) to measure the complexity of encoding and decoding with a
particular method.  This is done explicitly 
for codes based on the groups $\G(r,1,n)$
in Section~\ref{gr1n}.
The use of subgroups and coset leaders allows us to break the
decoding process into parts of manageable size and there
are often natural
candidates for the subgroup sequence, perhaps more than one.
\emph{Efficiency dictates that 
the subgroup sequence should be chosen so as to make the index of
consecutive terms small.}
That statement may be vague, but the principle is not:  The efficiency
of encoding and decoding is roughly proportional to the sum of the indices
of the consecutive subgroups.
For at each stage of decoding, one must choose a coset leader $d_k$ from
a collection of $[\G_{k-1}:\G_k]$ possibilities.
Thus there are at most $\sum_{k=1}^n [\G_{k-1}:\G_k]$ steps to subgroup 
decoding, compared with $|\G| = \prod_{k=1}^n [\G_{k-1}:\G_k]$
steps needed to search through the whole group.

Together, these criteria give us a way to determine how well
a given coding scheme works.

%%%%%%%%%%%%%%%%%%%%%%%%%%%%%%%%%%%%%%%%%%%%%%%%%%%%%%%%%
%%%%%%%%%%%%%%%%%%%%%%%%%%%%%%%%%%%%%%%%%%%%%%%%%%%%%%%%%5

%%%%%%%%%%%%%%%%%%%%%%%%%%%%%%%%%%%%%%%%%%%%%%%%%%%%%5
%%%%%%%%%%%%%%%%%%%%%%%%%%%%%%%%%%%%%%%%%%%%%%%%%%%%
\section{Geometric notions of minimal coset representatives}
\label{geometricnotions}

We now identify conditions on coset representatives that will 
guarantee correct
subgroup decoding, with
some channel noise or without.
In standard subgroup decoding for Coxeter groups, 
coset leaders are determined algebraically.  
If $\H \leq \K$ represents a consecutive pair in the subgroup sequence,
then each coset leader $c$ is chosen as an element 
in the coset of
minimal length when written as a product of generators of $\K$.
When we use a sequence of parabolic subgroups and choose 
simple reflections as generators, 
a unique shortest length element exists in each coset.
The algebraic condition of minimal length
(in terms of simple reflections)
for real reflection groups then guarantees 
certain geometric properties advantageous for coding
(see~\cite{FNP}).  
We seek geometric analogs of minimal length coset representatives
for arbitrary (complex) isometry groups
that preserve a nested sequence of regions.

%%%%%%%%%%%%%%%%%%%%%%%%%%%%%%%%%%%%
\subsection{The fundamental region and decoding region}
We use an analog
of a fundamental domain containing $x_0$:
%%%%%%%%%%%%%%%%%%5
\begin{definition}
The \emph{fundamental region} of a subgroup $\H \leq \G$ comprises
one vector closest to $\x_0$ from each $\H$-orbit (when a unique
closest vector exists)
after ignoring the isotrophy subgroup of $\x_0$: 
\[ \op{FR}(\H) = \{ \x \in \V : \Vert \x- \x_0 \Vert 
  < \Vert h\x- \x_0 \Vert 
\text{ whenever } h \in \H - \op{Stab}_{\H}(\x_0) \}\, .  \]  
\end{definition}
%%%%%%%%%%%%%%%%%%%%%%%%%%%%%%
Thus, 
the vectors in the fundamental region $\op{FR}(\G)$ are precisely
those that decode to $I$ (or to a message equivalent to $I$)
under correct decoding. We likewise
define a decoding region for each group element $g$
to be the set of vectors that decode to $g$ (or any
message equivalent to $g$) under correct decoding
(with no ties, see Section~\ref{ties}):

%%%%%%%%%%%%%%%%%%%%%%%%%%%%%%%%%5
\begin{definition}
The {\em decoding region} of $g \in \G$ is the set of
vectors that are closer to codeword $g^{-1}\x_0$
than any other codeword:
\[ \op{DR}(g) = \{ \x \in \V : \Vert g\x - \x_0 \Vert < 
    \Vert a \x - \x_0 \Vert \text{ whenever } a \notin \S g \}   \] 
for $\S=\op{Stab}_{\G}(\x_0)$.
\end{definition}
%%%%%%%%%%%%%%%%%%%%%%%%%%%%%%%%%%%%%%%%%%%%%%%%%%%%%%%%%%

Thus an algorithm decodes robustly exactly when it decodes
every vector in $\op{DR}(g)$ to a group element equivalent to 
$g$.
Note that the decoding region for $g$ is just a translate
of the fundamental region for $\G$:
$$g\op{DR}(g) = \op{DR}(I) = \op{FR}(\G) 
\ .$$
Also note that the fundamental regions of subgroups of $\G$ are nested
in the reverse order:  If $\H \leq \K \leq \G$, then 
$\op{FR}(\H) \supseteq \op{FR}(\K)\supseteq \op{FR}(\G)$.  

%%%%%%%%%%%%%%%%%%%%%%%%%%%
\begin{remark}
{\em If $\x_0$ has full orbit, then
no vector in $\V$ fixed by a
nonidentity group element lies in a decoding region.  In particular,
if $\G$ is a real or complex reflection group,
the decoding regions exclude vectors on reflecting hyperplanes.
In fact, they give us an 
analog of (Weyl) {\em chambers}:
If $\G$ is a Coxeter group, then the
fundamental region is just a fundamental chamber that
contains $\x_0$ and 
the decoding region of $g$ in $\G$ is
just the chamber containing $g^{-1}\x_0
$.
}
\end{remark}
%%%%%%%%%%%%%%%%%%%%%%%%%%%%%%%%%%%%%

%%%%%%%%%%%%%%%%%%%%%%%%%%%%%%%%%%%%%%%%%%%%%%%%%%%%%%%%%%%%%%%5
\subsection{The initial vector and minimum distance}
The initial vector $\x_0$ determines the isotrophy subgroup 
$\S = \op{Stab}_{\G}(\x_0)$ and the {\em minimum distance} of the code
defined by
\[ d_{min} = \min_{b \notin \S a} \Vert a^{-1} \x_0 - b^{-1} \x_0 \Vert
= \min_{a \notin \S} \Vert a \x_0 - \x_0 \Vert. \] 
As with any coding scheme, a large minimum distance is desirable. 
However, it turns out that for complex reflection groups, an initial vector 
that maximizes the minimum distance almost surely fails to satisfy some other 
important property, and in fact one must settle for a $d_{min}$ that 
is less than the maximum possible.
Note that if 
$\Vert \x - \x_0 \Vert < \frac 12 d_{min}$, then $\x \in \op{FR}(\G)$.

%%%%%%%%%%%%%%%%%%%%%%%%%%%%%%%%%%%%%%%%%%%%%%%%%%%%
\subsection{Minimal, Region Minimal, and Greed Compatible}

We now give geometric notions of minimal coset representative.
The following simple definition guarantees that a coset leader 
maps a fundamental region to a new region that at least contains $\x_0$.

%%%%%%%%%%%%%%%%%%%%%%
\begin{definition}
A coset leader $c$ for groups $\H\leq \K$
is {\em minimal} if 
$\x_0 \in c(\op{FR}(\H))$.
A set of coset leaders is {\em minimal} if all its elements are.
\end{definition}
%%%%%%%%%%
We will see in Section~\ref{minimalitysuffices} that minimal coset leaders
are both necessary and sufficient
for correct decoding (with some noise) when the 
initial vector has full orbit.
We need a stronger version of minimality though:
%%%%%%%%%%%%%%%%%%%%%%%%%%%%%%%%%%%%%%%%
\begin{definition}
A coset leader $c$ for groups $\H\leq \K$
is {\em region minimal}
if $\op{FR}(\K)\subseteq c(\op{FR}(\H))$.
A set of coset leaders is {\em region minimal} if all its elements are.
\end{definition}
%%%%%%%%%%%%%%%%%%%%%%%%%%%%%%%%%%%%%%%%%%%%%%%%
Note that a minimal or region minimal coset leader
may not exist because
two elements of the same coset may both yield the minimum
distance, creating a tie; see Section~\ref{ties}.

We interpret these two notions of minimality directly in terms
of finding a coset representative that minimizes 
distance back to the initial vector:
%%%%%%%%%%%%%%%%%%%%%%%%%%
\begin{lm}
A coset leader $c$ for groups $\H\leq \K$
is minimal if and only if 
it moves the initial vector $\x_0$ 
the least among 
other members of its coset (after excluding
the stabilizer subgroup of $\x_0$):
$$||c^{-1} \x_0 - \x_0|| < ||(ch)^{-1}\x_0 - \x_0||
$$
for all $h$ in $\H-\op{Stab}_{\H}(\x_0)$. 
\end{lm}
%%%%%%%%%%%%%%%%%%%%%%%%%%%%%%%%%%%%%%%%%%%%%%%%%%
\begin{lm}
A coset leader $c$ for groups $\H\leq \K$ is region minimal if 
and only if it maps the 
fundamental region $\op{FR}(\K)$ closer to the initial vector $\x_0$ than
other members of its coset, after inverting:
$$\text{For any } \y \text{ in } \op{FR}(\K),\quad
||c^{-1}\y-\x_0|| < ||(ch)^{-1} \y-\x_0||
$$
for all $h$ in $\H-\op{Stab}_{\H}(\x_0)$.
\end{lm}

The next definition offers a forward looking notion:
A set of coset leaders is compatible with the greedy algorithm if
every element in the larger fundamental region of $\H$ is 
sent into the smaller fundamental region of $\K$
by {\em some} coset leader:
%%%%%%%%%%%%%%%%%%%%%%%%%%%%%%%%%%%%%5
\begin{definition}
We call a set of coset leaders $\op{CL}$ for groups $\H\leq \K$
{\em \robust}
if there exists for every $\x \in \op{FR}(\H)$ a coset leader 
$c \in \op{CL}$ with $c\x \in \op{FR}(\K)$.
\end{definition}
%%%%%%%%%%%%%%%%%%%%%%%%%%%%%%%%%%%%

%%%%%%%%%%%%%%%%%%%%%%%%%%%%%%%%%%%%%%%%%%%%%%%%%%
We will see in Theorem~\ref{propAnne}
that if $\x_0$ is chosen with full orbit,
then region minimal representatives are \robust\ 
and vice versa.

%%%%%%%%%%%%%%%%%%%%%%%%%%%%%%%%%%%%%%%%%%%%%%%%%%%%%%%%%%
%%%%%%%%%%%%%%%%%%%%%%%%%%%%%%%%%%%%%%%%%%%%%%%%%%%%%%%%%%%%%%%%%%%%%%%%%%%%%%%%%%%%%%%%%%%%%%%%%%%%%%%%%%%%%%%%%%%%
\section{Greed pays....}\label{greedpays}
The subgroup decoding procedure uses a greedy algorithm, but
greedy algorithms don't always work:
The algorithm may not produce a group element
minimizing $\Vert a\r - \x_0 \Vert$ over {\em all} $a$ in $\G$. 
We now argue that
\robust\ coset leaders not only ensure that the subgroup decoding algorithm
will decode correctly, but that  
the algorithm is also robust.

For the remainder of the paper, the term \emph{subgroup sequence} will always
refer to a nested sequence of subgroups 
\[ \{ I \} = \G_0 < \G_1 < \ldots < \G_m=\G \ . \]

%%xxxxxxxxxxxxxxxxxxxxxxxxxxxxxxxxxxxxxxxxxxxxxxxxxxxxxxxx
\begin{thm}\label{algorithmworks}
Fix any
finite unitary group $\G$ acting on $\V$,
initial vector $\x_0$ in $\V$,
subgroup sequence, and 
coset leader sets $\op{CL}(\G_k/\G_{k-1})$.
If every set $\op{CL}(\G_k/\G_{k-1})$ is \robust, 
then the subgroup decoding algorithm decodes robustly
(and thus also correctly with some noise).
\end{thm}
%%%%%%%%%%%%%%%%%%%%%%%%%%%%%%%%%%%%%%%%%%%5
\begin{proof}
Assume a received vector $\r$ lies in $\op{FR}(g)$ for some $g$ in $\G$.
Inductively, $d_{k-1} \cdots d_1\r \in \op{FR}(\G_{k-1})$
and the algorithm chooses $d_k$ at the $k$-th stage
with $d_k\cdots d_1\r\in\op{GR}(\G_{k})$.
Thus $d_m \cdots d_1 \r$ is in the fundamental region of $\G$
and $\r$ decodes as $d_m\cdots d_1=g'$.
On the other hand, 
$\r = g^{-1}\x$ for some $\x \in \op{FR}(\G)$ since $\r \in \op{DR}(g)$.
Now $\x$ and $g'g^{-1}\x$ both lie in $\op{FR}(\G)$,
which implies (by the definition of fundamental region) that
$g'g^{-1} \in \S$. Thus $g' \in \S g$ and $g$ and $g'$ are equivalent.
Thus the subgroup decoding algorithm decodes robustly.
\end{proof}
%%%%%%%%%%%%%%%%%%%%%%%%%%%%%%%%%5

We will verify in Section~\ref{gr1n}
that \robust\ coset leaders exist for the complex reflection
groups $\G(r,1,n)$ 
for an appropriate subgroup sequence and initial vector.
In the next section, we show how to salvage 
correct decoding with small noise even when \robust\ group leaders
can not be found.

We point out in the next theorem
that if the initial vector $\x_0$ has full orbit,
then \robust\ coset leaders are region minimal and vice versa.

%%%%%%%%%%%%%%%%%%%%%%%
%\begin{lm} \label{ball}
%The following are equivalent.
%\begin{enumerate}
%\item $\op{Stab}_{\G}(\x_0) = \{ I \}$. 
%\item $\op{Stab}_{\G}(\x) = \{ I \}$ for all $\x \in \op{FR}(\G)$. 
%\end{enumerate}
%\end{lm}
%%%%%%%%%%%%%%%%%%%%%%%

%%%%%%%%%%%%%%%%%%%%%%%%%%%%%%%
\begin{thm}\label{propAnne}
Assume the initial vector $\x_0$ has full orbit.
A set of coset leaders is \robust\
if and only if it is region minimal.
\end{thm}
\begin{proof}
Assume that $\op{CL}=\op{CL}(\K/\H)$ is a \robust\ set
of coset representatives for $\K$ over $\H$ and take 
$\y \in \op{FR}(\K)$.  Let $c \in \op{CL}$.
Find $h$ minimizing
$\Vert hc^{-1}\y - \x_0 \Vert$ over all $h \in \H$.
Then $\x = hc^{-1}\y \in \op{FR}(\H)$, 
whence there is a coset leader $d \in \op{CL}$
such that $d\x \in \op{FR}(\K)$.  
As $\y$ and $dhc^{-1} \y$ both lie in $\op{FR}(\K)$, 
$dhc^{-1}=I$ and $dh=c$.  
Since $c$ and $d$ are both coset leaders, $c=d$ and $h=I$. 
Thus $c^{-1}\y \in \op{FR}(\H)$, as desired.

Conversely, assume that $\op{CL}$ is region minimal and
take $\x \in \op{FR}(\H)$.
Find $k$ minimizing $\Vert k\x - \x_0 \Vert$ over all $k\in \K$
and write $k=ch_0$ with $c\in\op{CL}$ and $h_0\in\H$.
Then $\y = k\x\in\op{FR}(\K)$, so 
$ c^{-1}\y = h_0\x\in\op{FR}(\H)$.  As both $\x$ and $h_0\x$
lie in $\op{FR}(\H)$ and $\x_0$ has full orbit, $h_0=I$.
Hence $c\x \in \op{FR}(\K)$
and $\op{CL}$ is \robust.  
\end{proof}
%%%%%%%%%%%%%%%%%%%%%%%%%%%%%%%%%%%%%%%%%%%%%%%%%%%%%%%%%%%%%%

%%%%%%%%%%%%%%%%%%%%%%%%%%%%%%%%%%%%%%%%%%%%%%%%%%%%%%%%%%%%%%%
\section{...but minimality suffices!} \label{minimalitysuffices}

We now turn to the case when $\x_0$ has full orbit under $\G$.
For example, we choose $\x_0$ off a reflecting hyperplane
if $\G$ is a real or complex reflection group.  
We show that minimality of induced coset leaders is
both necessary and sufficient for the 
 subgroup decoding algorithm to decode
correctly, even with some noise.  
We begin by defining {\em induced coset leaders} with
the following elementary lemma:

%%%%%%%%%%%%%%
\begin{lemma}
Fix sets  $\op{CL}(\G_k/\G_{k-1})$ 
of coset leaders for each consecutive pair 
in a subgroup sequence.
Then for any $k<\ell$, 
the set
$$\op{CL}(\G_{\ell}/\G_k)=\{c_{\ell} \cdots c_{k+1}:
c_i \in \op{CL}(\G_{i}/\G_{i-1}) 
\text{ for } k+1\leq i \leq \ell\}\ $$  
is a complete set of coset representatives for $\G_{\ell}$ over $\G_k$.
We call its elements 
the {\em induced coset leaders} for
$\G_{\ell}$ over $\G_k$.
\end{lemma}
%%%%%%%%%%%%%%%%%%%%%%%%%%%

We now give a necessary condition for correct decoding.
The next theorem explains that 
just as coset leaders are chosen to be the 
codewords of minimum Hamming weight in linear block coding, 
so too should coset leaders be chosen minimum
in a geometric sense in subgroup decoding.

%%%%%%%%%%%%%%%%%%%%%%%%%%%%%%%
\begin{thm}\label{minimalnecessary}
Fix any
finite unitary group $\G$ acting on $\V$,
initial vector $\x_0$ in $\V$ of full orbit,
subgroup sequence, and 
coset leader sets $\op{CL}(\G_k/\G_{k-1})$.
Minimal coset leaders are necessary for correct decoding:
If the subgroup decoding algorithm decodes
correctly, then
the induced coset leaders $\op{CL}(\G_{\ell}/\G_{k})$ are minimal
for all $k<\ell$.
\end{thm}
%%%%%%%%%%%%%%%%%5
\begin{proof}
Fix some index $k$ and suppose $c_k$ in $\op{CL}(\G_{k}/\G_{k-1})$
is not minimal.  Then there exists some nonidentity element
$h$ in $\G_{k-1}$ with
\begin{equation}\label{contradict}
||c_k^{-1}\x_0-\x_0||\geq ||(c_kh)^{-1}\x_0-\x_0||
=||(c_kc_{k-1} \cdots c_1)^{-1}\x_0-\x_0||,
\end{equation}
where $h=c_{k-1}\cdots c_1$ for some
$c_i$ in $\op{CL}(\G_i/\G_{i-1})$.
Fix some $j$ with $1\leq j\leq k-1$ and suppose
$\r_{j}=(c_kc_{k-1}\cdots c_j)^{-1}\x_0$ 
is a received vector.  As $\r_j$ correctly decodes
to group element $c_kc_{k-1}\cdots c_j$, the algorithm
chooses coset leader $c_j$ among all coset
leaders in $\op{CL}(\G_j/\G_{j-1})$ (including the coset leader
$I$) at the $j$-th step. Thus
$$
\begin{aligned}
|| (c_kc_{k-1}\cdots c_{j+1})^{-1}\x_0-\x_0||
&=||c_j(c_kc_{k-1}\cdots c_{j})^{-1}\x_0-\x_0||\\
&\leq ||I(c_kc_{k-1}\cdots c_{j})^{-1}\x_0-\x_0||\, .
\end{aligned}
$$
This gives a nested sequence of inequalities
as $j$ ranges from $1$ to $k-1$,
$$
||c_k^{-1}\x_0-\x_0||
\leq
||(c_kc_{k-1})^{-1}\x_0-\x_0||
\leq \ldots \leq
||(c_kc_{k-1}\cdots c_1)^{-1}\x_0-\x_0||
$$
with at least one inequality
strict as $h\neq I$,
contradicting inequality~(\ref{contradict}) above.
We replace $c_k$ by any $c_{\ell}c_{\ell-1}\cdots c_k$, 
where each $c_i$ lies in $\op{CL}(\G_i/\G_{i-1})$,
in the above argument to see
that induced coset leaders are minimal as well.
\end{proof}
%%%%%%%%%%%%%%%%%%%%%%%%%%%%%%%%%%%%%%%%%%%%%%%%%%%%%%%%%%%%%%%%55

In the last section, we saw that region minimal coset leaders
guarantee robust decoding (Theorem~\ref{algorithmworks}). 
However, it is not always easy to determine the fundamental region 
of a subgroup $\G_k$ in the subgroup sequence of a complicated group.
Even worse, region minimal coset leaders may fail to exist.
The next theorem shows that the decoding algorithm corrects for small noise when we weaken the hypothesis on coset leaders
but shrink the region of correct decoding to compensate.
We may merely insist that induced coset leaders be minimal, %
a condition which is straightforward to test but
fails for many choices of subgroup sequences
(see Section~\ref{othergroups}).

%%%%%%%%%%%%%%%%%%%%%%%%%%%%%%%%%
\begin{thm} \label{shrinkregion}
Fix any
finite unitary group $\G$ acting on $\V$,
initial vector $\x_0$ in $\V$ with full orbit,
subgroup sequence, and
coset leader sets $\op{CL}(\G_k/\G_{k-1})$.
If every set of induced 
coset leaders for $\G$ over $\G_k$
is minimal (for $1 \leq k < m$),
then the subgroup decoding algorithm decodes
correctly with some noise.
\end{thm}
%%%%%%%%%%%%%%%%%%%%%%%%%%%%%%
\begin{proof}
Let $\delta_m = d_{min}$, the minimum distance of the code, and 
for $1 \leq k < m$, define
\[  
\delta_k = \min \big\{ \| 
c_m \cdots c_{k+1}h \x_0 - 
\x_0 \| 
-
\| c_m \cdots c_{k+1}\x_0 - 
\x_0 \| \big\},
\]
taking the minimum over all 
$c_i \in \op{CL}(\G_i/\G_{i-1})$ for $k< i\leq m$
and over all $h \in \G_k - \G_{k-1}$.
Set $\delta = \min_{1 \leq k \leq m} \delta_k$.
Since each $\op{CL}(\G/\G_{k})$ is minimal,
each $\delta_k$ is nonzero and thus $\delta$ is nonzero.

Suppose $g$ in $\G$ is a message with transmitted vector
$g^{-1}\x_0$.  Write $g$ uniquely as
$g=c_m \cdots c_1$ with each $c_i$ in $\op{CL}(\G_i/\G_{i-1})$.
Assume the received vector $\r$ is within $\delta/2$
of the transmitted vector.  Then, for $\r_0 = \r$,
\[
\| c_1\r_0 - (c_m \cdots c_2)^{-1} \x_0 \| =
\| \r_0 - (c_m \cdots c_1)^{-1} \x_0 \| < \delta/2.
\]
By the triangle inequality,
\begin{align*}
\| c_1\r_0 - \x_0 \| &\leq \| c_1\r_0 - (c_m \cdots c_2)^{-1}\x_0 \| +
\| (c_m \cdots c_2)^{-1}\x_0 - \x_0 \| \\
&< \frac {\delta}2 + \| (c_m \cdots c_2)^{-1}\x_0 - \x_0 \| 
\end{align*}
while for $d \in \op{CL}(\G_1/\G_0) - \{ c_1 \}$,
\begin{align*}
\| d\r_0 - \x_0 \| &\geq  - \| d\r_0 - d(c_m \cdots c_1)^{-1}\x_0 \| +
 \| d(c_m \cdots c_1)^{-1}\x_0  - \x_0 \| \\
    &= - \| \r_0 - (c_m \cdots c_1)^{-1}\x_0 \| 
     +  \| (c_m \cdots c_1d^{-1})^{-1}\x_0  - \x_0 \| \\
    & > - \delta/2  +  \| (c_m \cdots c_2)^{-1}\x_0  - \x_0 \| + \delta_1 \\
    &\geq \delta/2 + \| (c_m \cdots c_2)^{-1}\x_0 - \x_0 \| 
\end{align*}
because $\delta \leq \delta_1$.  Hence the subgroup decoding algorithm, 
which chooses a coset leader $c$ minimizing $\| c\r_0 - \x_0 \|$,
will choose $c=c_1$.

Now let $\r_1 = c_1 \r_0$ and note that 
\[
\| \r_1 - (c_m \cdots c_2)^{-1} \x_0 \| =
\| \r_0 - (c_m \cdots c_1)^{-1} \x_0 \| < \delta/2.
\]
An analogous argument shows that the subgroup decoding algorithm will choose the coset leader $c_2$
(since the product of $c_2$ with the inverse of any other coset leader in
$\op{CL}(\G_2/\G_1)$ lies in $\G_2-\G_1$) at the second stage.  Recursively, 
the algorithm chooses $c_3, \ldots, c_{m-1}$ as
coset leaders minimizing distance to $\x_0$.
For the last step, we set $\r_{m-1} = c_{m-1} \cdots c_1 \r_0$ 
and note that
$$\| c_m \r_{m-1} - \x_0 \| < \delta/2\ $$
while for any other coset leader $d\in \op{CL}(\G_m/\G_{m-1})$,
$$\| d\r_{m-1}-\x_0\| 
\geq 
- \| d\r_{m-1}-dc_m^{-1}\x_0 \|+ \|dc_m^{-1}\x_0-\x_0\|
> - \delta/2 + d_{min}
\geq \delta/2.
$$
Hence the algorithm chooses $c_m$ as well and outputs
$g'=g$ as the decoded message.
\end{proof}

%%%%%%%%%%%%%%%%%%%%%%%%%%%%%%%%%%%%%%%%%%%
%%\begin{remark}
%%{\em
%%Coset leaders may not be minimal because ofties
%%(see Section~\ref{G_4})
%%yet we may still recover correct decoding.
%%We may weaken the hypothesis in Theorem~\ref{shrinkregion}
%%and merely assume that $\x_0$ lies 
%%in the fundamental region of $\G_j$ relative to $\G_{j-1}$
%%after moving it by any induced coset leader for $\G$ over $\G_j$
%%and yet maintain correct decoding.  
%%More precisely, let
%%$\op{FR}(\K,\H)$ be the {\em relative fundamental region}
%%$$
%%\op{FR}(\K,\H)
%%=\{\x\in \V: \| \x-\x_0\| < \|k\x-\x_0\| \text{ whenever }
%%k\in \K-\H\};
%%$$
%%if we merely require (for each $j$) that
%%induced coset leaders be {\em relatively minimum}, i.e.,
%%$$\x_0 \in (c_m\cdots c_{j+1})\op{FR}(\G_j,\G_{j-1})
%%\text{ for all } c_k \in \op{CL}(\G_k/\G_{k-1})\, , $$
%%then there exists $\delta>0$ so that the algorithm decodes correctly 
%%any received vector within $\delta$ of the transmitted vector.
% Typos in above remark....
%%\end{remark}
%%%%%%%%%%%%%%%%%%%%%%%%%%%%%%%%%%%%5
%%%%%%%%%%%%%%%%%%%%%%%%%%%%%%%%%%%%%%%%%%%%%%%

The last theorem together with Theorem~\ref{minimalnecessary}
now gives us necessary and sufficient conditions for
correct decoding:
%%%%%%%%%%%%%%%%%%%%
\begin{cor}\label{correctdecoding}
Choose an initial vector $\x_0$ with full orbit under $\G$.
Correct subgroup decoding occurs if and only if
induced coset leaders $\op{CL}(\G/\G_k)$ are minimal for all $k$. 
In this case, subgroup decoding decodes correctly 
with some noise.
\end{cor}
%%%%%%%%%%%%%%%%%%%%%%%%%%%%%%%%%%%%

One can prove directly or appeal to the last corollary to check
that very short subgroup sequences always decode correctly
with minimal coset leaders:
%%%%%%%%%%%%%%%%%%%%%%%%%%%%%%%%%%%%%%%%%%%%%%%%%5
\begin{cor}\label{shortseq}
Assume the initial vector has full orbit under $\G$.
Consider a short subgroup sequence $\{I\}<\G_1<\G$.
Then the subgroup decoding algorithm decodes correctly
(with some noise) if and only if the coset leaders for $\G$ over $\G_1$ are minimal.
\end{cor}

For example, this corollary applies to the octahedral reflection
group $\G_8$
of Section~\ref{G_8} with the natural subgroup sequence
$\{ I \} < \{ I,A,A^2,A^3 \} < \G_8$.  With an appropriate choice of
the initial vector, it is straightforward to find minimal coset leaders for $\G_8$.
Compare with Section~\ref{ties}, though, for difficulties
inherent in finding minimal coset leaders in general.
 
%%%%%%%%%%%%%%%%%%%%%%%%%%%%%%%%%%%%%%%%%%%%%%%%%%%%%%%%%%%%%%%
%%%%%%%%%%%%%%%%%%%%%%%%%%%%%%%%%%%%%%%%%%%%%%%%%%%%%%%%%%%%%%%
%%%%%%%%%%%%%%%%%%%%%%%%%%%%%%%%%%%%%%%%%%%%%%%%%%%%%%%%%%%%%%%
\section{Comparing Minimal, Region Minimal, and Greed Compatible}
\label{comparingnotions}
In this section, we make a few observations comparing
the different geometric notions of minimal coset representatives.
We begin by comparing region minimal with minimal:
%%%%%%%%%%%%%%%%%%%%%%%%%%%%%%%%%%%%%%%%%%%%%%%%%%%%%%%%%%%%%%%5
\begin{thm} \label{analyziz}
Fix any
finite unitary group $\G$ acting on $\V$,
initial vector $\x_0$ in $\V$ of full orbit,
subgroup sequence, and
coset leader sets $\op{CL}(\G_k/\G_{k-1})$.
Then the following (where $1\leq k,j\leq m$)
hold for induced coset leaders.
\begin{enumerate}
\item If\/ $\op{CL}(\G_k/\G_{k-1})$ is region minimal
then it is minimal. 
\item If\/ $\op{CL}(\G/\G_j)$ is region minimal then it is minimal.
\item If\/ $\op{CL}(\G_k/\G_{k-1})$ is region minimal for all $k$,\\
then $\op{CL}(\G/\G_j)$ is region minimal for all $j$.
\item If\/ $\op{CL}(\G_k/\G_{k-1})$ is minimal for all $k$,\\
then $\op{CL}(\G/\G_j)$ need not be minimal
  for all $j$.
\end{enumerate}
\end{thm}
The proofs of (1)--(3) in the last theorem
are straightforward using Theorem~\ref{propAnne}. 
The claim in (4) is shown with an example
based on the complex reflection group $\G_{25}$ given in Kim~\cite{HJK}.

Theorem~\ref{propAnne} 
then implies
%%%%%%%%%%%%%%%%%%%%%%%%%%%%%%%%%%%%%%%%%%%%%%%%%%%%%%%%%%%
\begin{cor}\label{howtochoose}
Assume the initial vector $\x_0$ has full orbit under $\G$.
Any \robust\ set of coset leaders is also minimal.
\end{cor}
%%%%%%%%%%%%%%%%%%%%%%%%%%%%%%%%%%%%%%%%%%%%%%%%%%%%%%%%%%%%

Recall that minimal coset leaders guarantee correct
decoding so long as noise remains under some
threshold (see Corollary~\ref{correctdecoding});
we give that thresh-hold explicitly
when coset leaders are \robust\ and sharpen the corollary
above:

%%%%%%%%%%%%%%%%%%%%%%%%%%%%%%%%%%%
\begin{thm}\label{rmanddmin}
Fix any finite unitary group $\G$ acting on $\V$,
initial vector $\x_0$ in $\V$ with full orbit, 
subgroup sequence, 
and coset leader sets $\op{CL}(\G_k/\G_{k-1})$.
If every set $\op{CL}(\G_k/\G_{k-1})$ is \robust, 
then the subgroup decoding algorithm decodes 
any received vector $\r$ satisfying
$\Vert \r - g^{-1}\x_0 \Vert < \frac 12 d_{min}$   
to the message $g$ in $\G$.
The corresponding statement is false if we replace 
$\frac 12 d_{min}$ by any $\gamma > \frac 12 d_{min}$.
\end{thm}
%%%%%%%%%%%%%%%%%%%%%%%%%%%%%%%%
\begin{proof}
If $\Vert \r - \x_0 \Vert < \frac 12 d_{min}$, then $\x \in \op{FR}(\G)$.
Hence by Theorems~\ref{algorithmworks} and \ref{propAnne}, 
the vector $\r$ will decode to the message $g$.
On the other hand, there exists $a \in \G$ such that 
$\| a^{-1}\x_0 - \x_0 \| = d_{min}$.  For any $\varepsilon$ with
$0 < \varepsilon < \frac 12$, let 
$\r = \x_0 +  (\frac 12 + \varepsilon)(a^{-1}\x_0 - \x_0)$.
Then $\| \r - I \x_0 \| = (\frac 12 + \varepsilon)d_{min}$,
but $\r$ decodes to $a$ since 
$\| \r - a^{-1}\x_0 \| = (\frac 12 - \varepsilon) d_{min}$.
\end{proof}

%%%%%%%%%%%%%%%%%%%%%%%%%%%%%%%%%%%%%%%%%%%%%%%%%%%%%%%%%%%%
%%%%%%%%%%%%%%%%%%%%%%%%%%%%%%%%%%%%%%%%%%%%%%%%%%%%%%%%%%%%
%%%%%%%%%%%%%%%%%%%%%%%%%%%%%%%%%%%%%%%%%%%%%%%%%%%%%%%%%%%%

\section{Ties}\label{ties}

Correct decoding requires minimal induced coset leaders by
Corollary~\ref{correctdecoding}, but they may not exist because of ties.
For any subgroup $\H$ of $\G$,
we say that a \emph{tie} occurs when the vectors 
encoding two or more elements
from the same coset of $\H$ yield the same minimum distance to the 
initial vector, i.e., when $a\H = b\H$ for some $a$ and $b$ in $\G$
with 
$\| a^{-1}\x_0 - \x_0 \| = \| b^{-1}\x_0 - \x_0 \|$
minimizing $\| c^{-1}\x_0-\x_0 \|$
over all $c$ in the coset $a\H$.
There are a couple of ways ties occur naturally.

The first is when $a$ and $a^{-1}$ lie in the same coset of $\H$ 
(i.e., $a^2 \in \H$)
and $\| a^{-1}\x_0 - \x_0 \| = \| a\x_0 - \x_0 \|$ yields a minimum.
One has little choice but to change the subgroup sequence in this case,
as shown below for the concrete code based
on the complex reflection group $\G_4$.

The second way occurs when
the initial vector $\x_0$ is real and $a$ and $b$ are both symmetric
unitary matrices, so that each has inverse equal to its conjugate,
with $ab$ and $a^{-1}b^{-1}$ in the same coset.
(Such matrices arise in the natural
reflection representations of some complex reflection groups
where real initial vectors are often
a convenient choice.)  Then
$\| ab\x_0 - \x_0 \| = \| a^{-1}b^{-1}\x_0 - \x_0 \|$, 
and this distance could be minimal over the coset.
In this case, replacing the initial vector by one that is
properly complex will eliminate the tie.
Again, see the example of $\G_4$ below.

Another way to resolve the problem of ties
is to allow multiple coset leaders
and multiple canonical forms.  This worked for the real reflection groups
$W\!D_n$ in~\cite{FNP}
but generally seems to become cumbersome rather quickly.

%%%%%%%%%%%%%%%%%%%%%%%%%%%%%%%%%%%%%%%%%%%%%%%%%%%%%%%%%%%%55
\subsection{Tetrahedral group $\G_4$}
\label{G_4}  

We give an example of a complex reflection group and choice
of initial vector and subgroup sequence for which minimal
coset leaders do not exist and thus the subgroup decoding
algorithm does not decode correctly. We then show how to make
other choices to recover correct decoding.
The group $\G_4$ of order $24$ (with $8$ reflections)
is generated by the matrices
\[
A = \begin{bmatrix}\ 1 & 0 \\ 
\rule{0ex}{4ex}
\  0 & -\frac 12 + \frac{\sqrt{3}}2 i 
\hphantom{x}\end{bmatrix}
\qquad\text{and}\qquad
B = \begin{bmatrix} 
\frac 1{\sqrt 3} i &  \frac 1{\sqrt 2} - \frac 1{\sqrt 6} i \\ 
\rule{0ex}{4ex}
\ \frac 1{\sqrt 2} - \frac 1{\sqrt 6} i & 
\ \frac 12 + \frac 1{2\sqrt 3} i 
\hphantom{x}\end{bmatrix}
\]
which satisfy $A^3=B^3=I$ and $ABA=BAB$.
As explained in Walker~\cite{CW}, an optimum choice of the initial vector
(for other decoding methods) is approximately $\x_0 = (0.8881, 0.4597)$.

Suppose we take the natural subgroup sequence
$\{ I \} < \{ I,A,A^2 \} < \G_4$.  
Set $C = BA^2B$ and $D=CA$, so that $C$ and $D$ are inverse
but in the same coset %$C\G_1 = D\G_1
$\{C,D,CA^2\}$ of $\G_1$, with
$$
||C\x_0-\x_0||
=
||D\x_0-\x_0||
<
||CA^2\x_0-\x_0||.
$$
Thus no minimal coset leader exists for this coset
because of a tie.
Note that $C^2=A^2\in \H$.

We could use instead the
subgroup sequence
$\{ I \} < \K < \G_4$ where $\K = \{ I,C,C^2,C^3,C^4,C^5 \}$.  
Then $\K$ has index four and 
minimal coset leaders for $\K$ are $I$, $B$, $B^2$. A tie
prevents choosing 
$AB$ or $A^2B^2=A^{-1}B^{-1}$ as a minimal
coset leader for the last coset.
(Here, $A$ and $B$ are symmetric unitary matrices.)
We resolve the tie by choosing a different 
initial vector.  Again consulting Walker~\cite{CW}, we choose
$\y_0 = (\frac 1{\sqrt{2}}+\frac i2, \frac 12)$
and $A^2B^2$ becomes the minimal coset leader.

Corollary~\ref{shortseq} implies
that the subgroup decoding algorithm decodes correctly with some noise  
for $\G_4$ 
with these revised choices.

%%%%%%%%%%%%%%%%%%%%%%%%%%%%%%%%%%%%%%%%%%%%%%%
\section{Efficient decoding using coset leader graphs} \label{type3}

We have thus far discussed mathematical properties that 
correct for noise.  Before considering control of errors in
the next section,
we turn our attention to matters of efficiency.
Throughout this section, we will assume that the initial
vector $\x_0$ has full orbit
under $\G$.
Given a fixed choice of coset leaders, the subgroup decoding algorithm
decodes by determining a coset leader 
at each step in the algorithm
that minimizes some distance.  Efficiency dictates that that we not loop
through {\em all} coset leaders in some fixed set $\op{CL}(\G_k/\G_{k-1})$ 
at each step.
One may instead use a restriction of
the standard Cayley graph 
to determine an appropriate choice.
(See Kriloff and Lay~\cite{KriloffLay} for an analysis
of Cayley graphs for $\G(r,1,n)$.) 
%%%%%%%%%%%%%%%%%%%%%%%%%%%%%%%%%%%%%%%%%%%%%%
\begin{definition}
Given a group $\G$ with subgroups $\H \leq \K$
and a set $X$ of generators 
for $\K$, the \emph{coset leader graph} $\Gamma=\Gamma(\K/ \H)$ 
for $\K$ over $\H$ 
with respect to a fixed set $\op{CL}(\K/\H)$ of coset leaders is
the graph \begin{itemize}
\item
whose vertices are the elements of $\op{CL}(\K/\H)$
\item
with a directed edge (labeled by $a$) from vertex $c$ to $d$ whenever $c = ad$ for 
some generator $a$ in $X$.
\end{itemize}
\end{definition}

Given a unitary group, a subgroup sequence, and an initial vector,
Theorem~\ref{minimalnecessary} tells us that the coset leaders
should be chosen minimal.  
If we also specify a generating set $X_k$ for each subgroup $\G_k$,
then the coset leader graphs are determined.

%%%%%%%%%%%%%%%%%%%%%%%%%%%%%%%%%%%%%%%%%%%%%5
\begin{definition} \label{propCL}
A set $\op{CL}(\K/\H)$ of coset leaders for $\K$ over $\H$ is
{\em connected} with respect to a fixed generating set $X$ of\/ $\K$ 
if its coset leader graph $\Gamma(\K/\H)$ is connected.
\end{definition}
%%%%%%%%%%%%%%%%%%%%%%%%%%%%%%%%%%%%%%%%%%%%%%5
Thus a set of coset leaders $\op{CL}$ is connected with respect to $X$ if
$c\in \op{CL}$ implies existence of a generator 
$a \in X$ and coset leader $d \in \op{CL}$ 
such that $c=ad$ or $c=a^{-1}d$.

For the groups considered in this paper, 
the coset leader graphs will be 
not only connected, but will be trees and cycles.  
Therefore we can safely ignore
some of the complications that
arise when navigating the more 
complex coset leader graphs associated with
exceptional reflection groups (see~\cite{FNP}).

An effective subgroup coding scheme may use
information in the coset 
leader graph to find the representation
of a group element as a product of coset leaders, $g = c_m \cdots c_1$; see~\cite{FNP}.  
A factorization of a coset leader $c$ for $\K$ over $\H$ into generators of $\K$ can 
be reconstructed by tracing a path from $I$ to $c$ in the coset leader graph and reading the 
edge labels in order.  Such a path need not be unique. 
However, for any finite group and connected coset leader graph, 
one may identify 
a canonical path by ordering the generators,
and this determines a spanning tree $T$ in $\Gamma$.

We add efficiency to the subgroup decoding algorithm
by specifying directions for navigating 
through the coset leader graphs $\Gamma_k$ for $\G_k$ over $\G_{k-1}$
to determine a coset leader at each stage of the algorithm.
We move
through the coset leader tree downward from
its root at $I$, each step getting closer to the initial vector, ceasing 
when steps take us further away again.  It is desirable in a group 
code that $\Vert g\x_0 - \x_0 \Vert$ be roughly proportional to the (minimum) 
length of $g$ as a word in generators and their inverses (for each $g$ in $\G$).
However, there may be ways to improve
the process.  For example, one can use a shortcut for rotations,
where the coset leader graph is a cycle.

%%%%%%%%%%%%%%%%%%%%%%%%%%%%%%%%%%%%%%%%%%%%%%%%%%%%%%%%%%%%%%%%%%%%%
%%%%%%%%%%%%%%%%%%%%%%%%%%%%%%%%%%%%%%%%%%%%%%%%%%%%%%%%%%%%%%%%%%%%%
%%%%%%%%%%%%%%%%%%%%%%%%%%%%%%%%%%%%%%%%%%%%%%%%%%%%%%%%%%%%%%%%%%%%%
\section{Error control}\label{errorcontrol}
We now consider error control.  Decoding errors 
occur with substantial noise:
A codeword $g^{-1}\x_0$
may be sent (for some message $g$ in $\G$) but the received vector $\r$
may land closer to some other
codeword.
If $\r$ does not lie in the decoding region $\op{DR}(g)$, 
then it most likely
lies in a geometrically neighboring decoding region 
$\op{DR}(g')$.
Can we fine-tune the subgroup decoding algorithm so
the decoded message $g'$ is close to the sent message $g$ most of the time?
In this section, we give properties of a group code that
ensure that 
the decoded message will differ from the sent message in at most 
one factor when written as a product of coset leaders,
provided the received vector lands in a region neighboring 
the intended one.
We again assume the initial vector $\x_0$ has full orbit
throughout this section.

\begin{definition}
The \emph{nearest neighbors} of a codeword $\u$ are
the codewords $\v$ with $\Vert \u - \v \Vert = d_{min}$.
\end{definition}
%%%%%%%%%%%%%%%%%%%%%%%%%%%%
It is not difficult to see how nearest neighbors of the initial
vector determine nearest neighbors of any codeword:
%%%%%%%%%%%%%%%%%%%%%%%%%%%%%%%%%%%
\begin{lm} \label{nnlemma}
For all $g$ in $\G$,
the nearest neighbors of $g \x_0$ are the codewords
$g \w$ with $\w$ a nearest neighbor of\/ $\x_0$.  
\end{lm}
%%%%%%%%%%%%%%%%%%%%%%%%%%%%%%%%%%%%%%%%%%

It is useful to identify the group elements
yielding nearest neighbors.
%%%%%%%%%%%%%%%%%%%%%%%%%%%%%%%%%%%%5
\begin{definition}\label{nnsetdef}
The {\em neighborhood} $N_\G(\x_0)$ of\/ $\x_0$ is the set 
of nearest neighbors of\/ $\x_0$, i.e., the 
points in the orbit of\/ $\x_0$ closest to $\x_0$.
We say the corresponding set $N_\G$ of group elements {\em realizes} the neighborhood:  
\begin{align*}
N_{\G}(\x_0) &= \{ \v \in \G\x_0 - \{ \x_0 \} : \Vert \v - \x_0 \Vert = d_{min} \},\\
N_{\G} &= \{a \in \G : a\x_0 \in N_{\G}(\x_0)\}. 
\end{align*}
\end{definition}
Neighborhoods can be analogously defined for
any subgroup $\G_k$ in the subgroup sequence.
In the case that $\G$ is a Coxeter group, 
a set of simple reflections realizes the neighborhood of
$\x_0$.
More generally, the generators for each subgroup may be taken to
be a subset of simple reflections
so that the group elements realizing
the neighborhood for $\G_k$ generate $\G_k$ (see~\cite{FNP}).
We seek a similar property for general group codes below.

By Lemma~\ref{nnlemma}, if a codeword $g^{-1}\x_0$ is decoded incorrectly,
it will most likely be decoded as a neighbor $(bg)^{-1}\x_0$
with $b$ in $N_\G$.  To minimize the message error, we would like
the canonical form of $bg$ as a product of coset leaders
to differ as little as possible from that of $g$.
That is the effect 
of the next two error control properties for consecutive 
subgroups in the subgroup sequence, both from~\cite{FNP}.
Note that the first property depends on the choice
of the initial vector $\x_0$.
For any subset $X$ of $\G$, write $X^{-1}$
for the set $\{a^{-1}: a\in X\}$.

%%%%%%%%%%%%%%%%%%%%%%%%%%%%%%%%%%%%%%%%%%
\begin{prop}[Nearest Neighbors] \label{naygen}
The {\em Nearest Neighbors Property} holds for
a fixed set $X_{\G}$ generating $\G$ whenever
$N_\G \subseteq X_{\G} \cup X^{-1}_{\G}$.
\end{prop}
%%%%%%%%%%%%%%%%%%%%%%%%%%%%%%%%%%%%%%%%%%%%%

%%%%%%%%%%%%%%%%%%%%%
\begin{prop}[Error Control] \label{propWes}
Let $\H < \K$ be subgroups of $\G$ with a fixed set of
coset leaders $\op{CL}(\K/\H)$.
The {\em Error Control Property} holds
for sets of generators
$X_{\H}$ of\/ $\H$ and $X_{\K}$ of\/ $\K$ whenever
$$bc \in \op{CL}(\K/\H) \text{ or }
c^{-1}bc \in X_{\H} \cup X_{\H}^{-1}\ 
$$
for all $b \in X_{\K} \cup X_{\K}^{-1}$ and
$c \in \op{CL}(\K/\H)$.
\end{prop}
%%%%%%%%%%%%%0%%%%%%%
Note that the property implies
that either $bc$ is the coset leader for the coset
$bc\H$, or $c$ is the coset leader for $bc\H$ 
because $bc$ and $c$ lie in the same coset.

The Error Control Property minimizes small errors:

\vspace{1ex}

%%%%%%%%%%%%%%%%%%%%%%%%%%%%%%%%%%%%%%%%%%%%%%%%%%%%%%%%%%5
\begin{thm} \label{thmWes}
Assume that Error Control Property~\ref{propWes} holds 
for consecutive pairs
of a subgroup sequence
$
\{I\}=\G_0<\G_1<\cdots < \G_{m}=\G
$,
some choice of generators $X_{\G_k}$ of $\G_k$,
and some choice of coset leaders $\op{CL}(\G_k/\G_{k-1})$.
Suppose $g$ in $\G$ has
canonical form as a product of coset leaders given by
$$g=c_m \cdots c_1, \ \text{ each } c_k\in \op{CL}(\G_k/\G_{k-1}).$$ 
Then for any $b \in X_{\G} \cup X_{\G}^{-1}$,
the canonical form of $bg$ is 
$$bg=c_m' \cdots c_1',\ \text{ each } c_k'\in \op{CL}(\G_k/\G_{k-1}),$$ 
where $c_i'=c_i$ for all but one $i$.  In addition, for
that single index $j$ with $c_j'\neq c_j$, the coset leader $c_j'$ is 
adjacent to $c_j$ in the coset leader
graph for $\G_j$ over $\G_{j-1}$.
\end{thm}
%%%%%%%%%%%%%%%%%%%%%%%%%%%%%%%%%%%%%%
\begin{proof}
We proceed by induction on $m$.  If $m=1$, then every element is
a coset leader, and the conclusion is trivial. Let $m>1$.
Consider 
$g = c_m \cdots c_1$ and take $b \in X_{\G_m}=X_{\G}$.
If $bc_m$ is a coset leader, then $bg = (bc_m)c_{m-1} \cdots c_1$ is in canonical
form. 
If not, Property~\ref{propWes} implies that
\begin{align*}
bg &= c_m (c_m^{-1}bc_m) c_{m-1} \cdots c_1 \\
   &= c_m (b' c_{m-1} \cdots c_1)
\end{align*}
with $b' \in X_{\G_{m-1}}$ and we apply the induction hypothesis.
\end{proof}
%%%%%%%%%%%%%%%%%%%%%%%%%%%%%%%%%%%%%%%%%%%%%%%

The Error Control Property and Nearest Neighbors Properties
together imply nice error control:

%%%%%%%%%%%%%%%%%%%%%%%%%%%%%%%%%%%%%%%%%%%%%%%%%%%%%
\begin{cor}\label{controlnoise}
Assume Error Control Property~\ref{propWes} and
Nearest Neighbors Property~\ref{naygen} hold for
$\G$ with fixed subgroup sequence, initial vector, coset leader sets,
and generating sets $X_{\G_k}$ of $\G_k$.
Assume coset leaders are greed compatible.
If a received vector lies in the decoding region 
containing a nearest neighbor of $g^{-1}\x_0$
due to noise, then the subgroup decoding
algorithm decodes it to a group element
differing from $g$ in only one factor when written
as a product of coset leaders.
\end{cor}
%%%%%%%%%%%%%%%%%%%%%%%%%%%%%%%
\begin{proof}
Lemma~\ref{nnlemma} implies that the received vector
$\r$ lies in the decoding region of
 $g^{-1}b\x_0$ for some $b$ in $\G$ with $b\x_0$
a nearest neighbor of $\x_0$.
Theorem~\ref{algorithmworks} then implies that the subgroup decoding algorithm
will correctly decode $\r$ to $b^{-1}g$.
But Property~\ref{naygen} implies that $b$ or $b^{-1}$ lies in $X_{\G}$,
and hence $b^{-1}g$ differs from $g$ in only one factor
by Theorem~\ref{thmWes}.
\end{proof}
%%%%%%%%%%%%%%%%%%%%%%%%%%%%%%%%%%%%%%%%%%%%%%%%%%%%%%%%%%%%5
%%%%%%%%%%%%%%%%%%%%%%%%%%%%%%%%%%%%%%%%%%%%%%%%%%%%%%%%%%%%5
\section{Observations about the initial vector} \label{initvec}

Mittelholzer and Lahtonen~\cite{ML} gave an elegant and simple solution to
the problem of choosing the initial vector in the case 
$\G$ is a Coxeter group:
Any unit vector in the 
fundamental region can be taken for the initial vector, some work better
than others, and there is a straightforward algorithm to find the optimal
choice.  The geometry of 
arbitrary groups acting on complex space prevents
a clean generalization,
although the following simple observations can
be useful.

%%%%%%%%%%%%%%%
\begin{lm} \label{max-x0}
Fix an initial vector $\x_0$. 
\begin{enumerate}
\item  If $c$ is a complex number with $|c|=1$, and $\y_0 = c\x_0$,
then the code $\G\y_0$ has the same minimum distance as $\G\x_0$.
The nearest neighbors of $\y_0$ are the vectors $a\y_0$ with $a \in N_\G$.
\item If $h \in \G$ and $\z_0 = h\x_0$,
then the code $\G\z_0$ also has the same minimum distance as $\G\x_0$.
In this case, the nearest neighbors of $\z_0$ are the vectors $b\y_0$ with
$b \in hN_\G h^{-1}$.
\end{enumerate}
\end{lm}
%%%%%%%%%%%%%%%%%%%%

The first part of the lemma
suggests that the first entry of $\x_0$ may be taken to be
real (or imaginary), which can be useful.
Although we often choose the initial
vector $\x_0$ to be a 
\emph{real} unit vector, note that
occasionally
it is crucial for the minimality of coset
leaders that the initial vector \emph{not} be real.  
In either case, we usually
adjust the entries to make neighbors realized by a
preferred set of generators (for example, reflections).  
The preceding lemma gives us some guidance in making these adjustments.
%%%%%%%%%%%%%%%%%%%%%%%%%%%%%%%%%%%%5
%%%%%%%%%%%%%%%%%%%%%%%%%%%%%%%%%%%%%%%%%%%%%%%
\section{Decoding with wreath products}\label{wreath}
In this section, we consider some wreath products 
that act as isometries on finite dimensional complex
space and show that a natural subgroup sequence
and choice of coset leaders produce codes that
not only decode correctly, but also robustly.
We will apply the results to the 
infinite family $\G(r,1,n)$ of complex reflection groups
in Section~\ref{gr1n}.

Let $\H\subset \text{GL}_m(\CC)$
be a finite unitary group acting on the vector space $\CC^m$.
Let $\G$ be the wreath product of $\H$ with the symmetric group
$\Sym_n$,
$$
\G=\H \wr \Sym_n = \Sym_n \ltimes \H^n.
$$
Then $\G$ acts on $\V=\CC^{mn}$ as the unitary group of
all $mn\times mn$ block permutation matrices
with each block a matrix in $\H$.
We adopt a standard left notation for wreath products 
and write each element of $\G$ 
as the product of a permutation in $\Sym_n$ and
an $n$-tuple of matrices from $\H$,
$$
\G=\{\sigma(h_1, \ldots, h_n): h_i\in \H, \sigma \in \Sym_n\},
$$
so that $g(x_1, ...., x_n)
=(h_{\sigma(1)} x_{\sigma(1)}, \ldots, h_{\sigma(n)}x_{\sigma(n)})$
for $g=\sigma^{-1}(h_1, \ldots, h_n)$, where each $x_i$ lies in $\CC^m$.
   
Define a subgroup sequence 
$$
\{I\}=\G_0<\G_1<\cdots < \G_{2n-1}=\G
$$
by setting 
$$
\begin{aligned}
\G_{2\ell-1} &= \{\sigma(h_1, \ldots, h_\ell, 1, \ldots, 1):
\sigma\in \Sym_{\ell}\} 
&\text{ for } \ell=1,\ldots, n,\\
\G_{2\ell} & = \{\sigma(h_1, \ldots, h_{\ell+1}, 1, \ldots, 1):
\sigma\in \Sym_{\ell}\} 
&\text{ for } \ell=1,\ldots, n-1\\
\end{aligned}
$$
(viewing $\Sym_\ell$ as a subset of $\Sym_n$)
so that the subgroups $\G_k$ give block diagonal matrix groups:
$$
\begin{aligned}
\G_{2\ell-1} &= (\H\wr\Sym_\ell)\oplus \{I_{m(n-\ell)}\}\
&\text{ for } \ell=1,\ldots, n,\\
\G_{2\ell} & = (\H\wr\Sym_\ell)\oplus \H \oplus \{I_{m(n-\ell-1)}\}\
&\text{ for } \ell=1,\ldots, n-1,\\
\end{aligned}
$$
with
$I_k$ the $k\times k$ identity matrix.
An obvious choice of coset leaders for pairs of consecutive subgroups
arises. We select block diagonal matrices with one block from $\H$
and the rest the identity or we choose 
cycles in the symmetric group ending at a fixed index:
Set 
$$
\begin{aligned}
\op{CL}(\G_{2 \ell}/ \G_{2\ell -1})
&=
\{ (1,\ldots, 1, h, 1, \ldots, 1): h \in \H\ \text{ in the
$(\ell+1)$-th slot}\}, \\
\op{CL}(\G_{2 \ell+1}/ \G_{2\ell})
&=
\{\text{} (j\ j+1\ \ldots\ \ell +1) 
\in \Sym_{\ell+1}: 1\leq j\leq\ell+1 \}\, .
\end{aligned}
$$
Fix a unit vector $\v_0$ in $\CC^m$ suitable for $\H$,
i.e., so that a unique element
$h$ in $\H$ minimizes $|| h\v_0 -\v_0||$.
Extend $\v_0$ to a initial vector $\x_0$ for $\G$
by setting $\x_0=(u_1\v_0,u_2\v_0,\ldots, u_n\v_0)$
in $\V$ for some real numbers
$u_i$ with $0<u_1<\ldots< u_n$ such that $\x_0$ has unit length.

%%%%%%%%%%%%%%%%%%%%%%%%%%%%%%5
\begin{thm}\label{wreathgreedy}
Let\/ $\H$ be a finite unitary group and let $\G=\H\wr \Sym_n$
with the above subgroup sequence and
initial vector.
The above choice of coset leaders 
is greed compatible.
\end{thm}
\begin{proof}
We first note conditions 
that minimize a distance $||g\x - \x_0||$
over $g$ in $\G$.
Fix $\x=(x_1, \ldots, x_n)$ in $\V=\CC^{mn}$ with
each $x_i$ in $\CC^m$
and write an arbitrary $g$ in $\G$ as a product
$\sigma^{-1}(h_1, \ldots, h_n)$ with each $h_i$ in $\H$
and $\sigma$ in $\Sym_n$.
Then
$$
||g\x -\x_0||^2=
||\x_0||^2+||\x||^2 
- 2 \sum_{1\leq j\leq n} u_j \text{Re}\left(
\v_0^H h_{\sigma(j)}x_{\sigma(j)}\right)  %%\, , 
$$
where the superscript $H$ denotes conjugate transpose.
The distance $||g\x-\x_0||^2$ is minimal 
when the summation over $j$ in the last expression is maximal.
But recall that for any two strictly increasing sequences
of positive real numbers
$0 < \alpha_1 < \ldots < \alpha_k$ and $0 < \beta_1 < \ldots < \beta_k$, 
the sum $\sum a_jb_{\tau(j)}$ 
is maximized over all $\tau$ in $\Sym_k$
by $\tau = I$.
Hence $||g\x-\x_0||$ is minimal
over all $g$ in $\G$  when 
\begin{enumerate}
\item[(a)] 
$h_i$ maximizes $\text{Re}(\v_0^H h_i x_i)$ 
over all elements in $\H$ for $i=1,\ldots, n$,
and 
\item[(b)] $\sigma$ in $\Sym_n$ is chosen so that 
$$
\text{Re}\left(\v_0^H h_{\sigma(1)} x_{\sigma(1)}\right)
\leq \ldots \leq
\text{Re}\left(\v_0^H h_{\sigma(n)} x_{\sigma(n)}\right).
$$
\end{enumerate}
Note that if each $h_i$ in $(a)$ above is unique
and the inequalities in $(b)$ are strict,
then a unique group element $g$ minimizes
$||g\x-\x_0||$.
We apply this observation to the subgroups $\G_k$
in the subgroup sequence and conclude that
$$
\begin{aligned}
\op{FR}(\G_{2\ell-1})&=
\{(w_1,\ldots, w_n):  w_i \in \CC^m,\
\text{Re}\left(\v_0^H w_1\right)<\ldots < 
\text{Re}\left(\v_0^H w_{\ell}\right),\\
& \hphantom{aaa}
\text{Re}\left(\v_0^H w_i\right)>
\text{Re}\left(\v_0^H hw_i\right)
\text{ for all }
I\neq h \in \H, 1\leq i\leq \ell\},\\
\op{FR}(\G_{2\ell})&=
\{(w_1,\ldots, w_n):  w_i \in \CC^m,\
\text{Re}\left(\v_0^H w_1\right)<\ldots < 
\text{Re}\left(\v_0^H w_{\ell}\right),\\
& \hphantom{aaa}%\text{ and }
\text{Re}\left(\v_0^H w_i\right)>
\text{Re}\left(\v_0^H hw_i\right)
\text{ for all }
I\neq h \in \H, 1\leq i\leq \ell + 1\}.
\end{aligned}
$$

Suppose
$\x$ lies in some $\op{FR}(\G_{2\ell})$ 
and choose the unique coset leader $d$
from $\op{CL}(\G_{2\ell+1}/ \G_{2\ell})$ 
with  $d\x=(x_{\sigma(1)}, \ldots, x_{\sigma(n)})$
and
$$\text{Re}(\v_0^H x_{\sigma(1)})<\ldots<
\text{Re}(\v_0^H x_{\sigma(\ell +1)}).$$ 
Then $d\x$ lies in $\op{FR}(\G_{2\ell +1})$.
Now suppose $\x$ instead lies in $\op{FR}(\G_{2\ell-1})$ 
and choose the unique 
element $h$ in $\H$ maximizing 
$\text{Re}(\v_0^H h x_{\ell+1})$.  
Let $d$ in $\op{CL}(\G_{2\ell}/ \G_{2\ell-1})$ be the corresponding
coset leader (i.e., $d=I_{\ell}\oplus h \oplus I_{nm-\ell-1}$).
Then $d\x$ lies in $\op{FR}(\G_{2\ell})$.
Hence each 
$\op{CL}(\G_k/\G_{k-1})$ is a set of \robust\ coset leaders
for $k$ even or odd.
\end{proof}

Theorem~\ref{algorithmworks} implies that the subgroup decoding algorithm
decodes wreath product codes robustly:

%%%%%%%%%%%%%%%%%%%%%%%%%%%%%%%%%%%%%%%%%%%%%%%%5
\begin{cor}\label{wreathdecodescorrectly}
Let\/ $\H$ be a finite unitary group and let $\G=\H\wr \Sym_n$ 
with the above natural choice of subgroup sequence,
coset leaders, and initial vector.
Then the subgroup decoding algorithm decodes robustly.
\end{cor}
%%%%%%%%%%%%%%%%%%%%%%%%%%%%%%%%%%%%%%%%%%%%%%%5

We now investigate error control for wreath products.
We fix a 
set of generators  $X_k$ for each subgroup $\G_k$ in the subgroup
sequence: If $k$ is odd, we 
choose block diagonal matrices that are the identity except first block from $\H$
together with a set of consecutive transpositions in $\Sym_n$;
if $k$ is even, we add on block diagonal matrices that are the 
identity except for
a single block from $\H$.
Set
$$
\begin{aligned}
X_{2\ell-1}&=\{ (h,1,\ldots, 1), h\in X_{\H}\}
\cup 
\{(1\ 2), (2\ 3), \ldots, (\ell -1\ \ell)\}, \\
X_{2\ell}&=X_{2\ell -1} 
\cup
\{(1, \ldots, 1, h, 1\ldots, 1): h \in \H
\text{ in the $(\ell+1)$-th slot} \}\, .
\end{aligned}
$$

With these choices, we have good error control:
%%%%%%%%%%%%%%%%%%%%%%%%%%%%%%%%%%%%%%%%%%%%%%%%
\begin{proposition}\label{errorcontrolwreath}
Let\/ $\H$ be a finite unitary group and let\/ $\G=\H \wr \Sym_n$. 
The above natural choice of subgroup sequence,
coset leaders, initial vector, and
generators for each subgroup
in the subgroup sequence satisfies Error Control Property~\ref{propWes}.
\end{proposition}
\begin{proof}
Fix a pair of nested subgroups with smaller group of odd index,
say
$\G_{2\ell-1} < \G_{2\ell}$.  Take any $b$ in $X_{2\ell}$
and any $c=(1, \ldots, 1, h, 1, \ldots, 1)$
in $\op{CL}(\G_{2\ell}/\G_{2\ell-1})$,
with $h \in \H$.
If $b$ lies in $X_{2\ell-1}$, then $b$ and $c$ commute
and $c^{-1}bc=b\in X_{2\ell-1}$.
If $b\notin X_{2\ell-1}$, then $bc \in \op{CL}(\G_{2\ell}/\G_{2\ell-1})$.
Thus Error Control Property~\ref{propWes} is satisfied.

Now fix a pair of nested subgroups with smaller group of even index,
say
$\G_{2\ell} < \G_{2\ell+1}$.  Take any $b$ in $X_{2\ell+1}$
and any $c=(j\ j+1\ \ldots\ l+1)$ 
in $\op{CL}(\G_{2\ell+1}/\G_{2\ell})$.
First suppose $b=(h,1,\ldots, 1)$ with  $h$ in $\H$.
If $j>1$, then $c^{-1}bc=b\in X_{2\ell}$ (as $c$ and $b$ commute), while
if $j=1$, then $c^{-1}bc = (1,\ldots, 1, h, 1, \ldots, 1)
\in X_{2\ell}$.
Now suppose that $b=(i-1\ \ i)$ for some $i\leq \ell+1$.
If $i<j$, then $c^{-1} b c=b\in X_{2\ell}$ as $c$ and $b$ commute;
if $i=j$, then $bc$ is the coset leader
$(i-1\ \ldots\ \ell+1) \in \op{CL}(\G_{2\ell+1}/\G_{2\ell})$;
if $i=j+1$, then $bc$ is the coset leader
$(i\ \ldots\ \ell+1) \in \op{CL}(\G_{2\ell+1}/\G_{2\ell})$;
and if $j+1<i$,
then $c^{-1}bc=(i-2\ \ i-1)\in X_{2\ell}$.
Thus Error Control Property~\ref{propWes}  
is satisfied in this case as well.
\end{proof}

Theorem~\ref{thmWes} then implies that errors can be controlled
when they occur:

%%%%%%%%%%%%%%%%%%%%%%%%%%%%%%%%%%%%%%%%%%%%%%
\begin{cor}
Let\/ $\H$ be a finite unitary group and let $\G=\H\wr \Sym_n$ 
with the above natural choice of subgroup sequence,
coset leaders, initial vector, and generating sets
$X_k\subset \G$ for each $\G_k$.
Assume the Nearest Neighbors Property~\ref{naygen} holds.
If a received vector lands in the decoding region 
containing a nearest neighbor of $g^{-1}\x_0$
due to noise, then the subgroup decoding
algorithm decodes it to a group element
differing from $g$ in only one factor when written
as a product of coset leaders.
\end{cor}

%%%%%%%%%%%%%%%%%%%%%%%%%%%%%%%%%%%
\begin{remark}
\em{
One may interpolate a sequence of subgroups of $\H$ to 
refine the above process and improve the decoding efficiency.  
At the even stages, one could splice a fixed subgroup sequence for $\H$ 
into the $(l+1)$-st coordinate and replace
$\G_{2\ell}$ with a new sequence.
One should take robust coset leaders for the subgroup
sequence of $\H$ 
and fix generators satisfying the Error Control Property~\ref{propWes}
for $\H$ so that the wreath product
$\G=\H\wr \Sym_n$ with the refined subgroup sequence
would also inherit robust decoding with error control.
But one could also use other methods to decode $\H$ at the even steps.  
That is the process envisioned in the decoding of wreath products
in Nation and Walker~\cite{JBNCW}, where the Snowflake Algorithm is used
to decode $\H$ at the even steps.
}
\end{remark}

%%%%%%%%%%%%%%%%%%%%%%%%%%%%%%%%%%%%%%%%%%%%%%%%%%%%%%%%%%%%%%%%
%%%%%%%%%%%%%%%%%%%%%%%%%%%%%%%%%%%%%%%%%%%%%%%%%%%%%%%%%%%%%%%%
\section{Unitary groups and reflection groups} \label{reflections}

The set of all $n \times n$ complex unitary matrices forms a group 
$\mathbb U(n)$, and the various groups we use for coding
are contained in its infinite subgroup
of monomial 
matrices (i.e., those with a single nonzero entry in each row
and in each column)
whose nonzero entries have norm 1.
If $r\geq 1$ is an integer, 
the group $\G(r,1,n)$ 
consists of monomial $n\times n$
matrices whose nonzero entries 
are $r$-th roots
of unity.  For any integer $p$ dividing $r$, the 
group $\G(r,p,n)$ consists of those matrices in $\G(r,1,n)$
whose nonzero entries 
multiply to an $(r/p)$-th root of unity.
For example, $\G(2,2,n)$ is the real Coxeter group $W\! D_n$.

A \emph{reflection} on a real or complex vector space
is a non-identity linear transformation
that fixes a hyperplane in that space pointwise.
Every reflection $s$ satisfies
   \[ s(\x)=\x + l_H(\x) {\boldsymbol{\alpha}} \text{ for all }\x \in \V \]
for some fixed vector $\boldsymbol{\alpha}$ in $\V$
and some
linear form $l_H$ in the dual space $\V^*$ that defines the reflecting
hyperplane $H$ fixed by $s$ (i.e., $\ker l_H=H$).
If $s$ is an isometry (for example, if $s$ has finite order),
then $s$ is the diagonal matrix
$\text{diag}(\lambda, 1, \ldots, 1)$ with respect to 
some basis of $V$ with
$\lambda=\det(s)$ the nonidentity eigenvalue of absolute value 1.
(In particular, $s$ has finite order if and only if $\lambda$
is a root-of-unity.)
In this case, 
we may choose $\boldsymbol{\alpha}$ to be a vector perpendicular 
to $H$ (with respect to an $s$-invariant inner product $\langle\ ,\ \rangle$
on $\V$) of length one and choose $l_H$ to be the function
$$
l_H(\x)=(\lambda-1) \langle \alpha,\x \rangle
\text{ for all } \x \in \V\, .
$$
If $s$ is a reflection on a real vector space, 
then $\lambda=-1$, and $s$ is an involution.  

A complex \emph{reflection group} is a group generated by a set of 
reflections on $\V=\CC^n$.  We assume all reflection groups are finite
and thus unitary with respect to the standard inner product. 
Note that every real
reflection group defines a complex reflection group
after extending scalars.
The finite irreducible complex reflection groups were classified in
a classic paper of Shephard and Todd~\cite{ST}:
%%%%%%%%%%%%%%%%%%%%%%%%%%%%%%%%%%%%%%%%%%%%%%%%%%%
Every finite irreducible complex reflection group is 
\begin{enumerate}
   \item $\G(r,p,n)$ for some $r,p,n\geq 1$ with $p$ dividing $r$, or
   \item one of the exceptional groups denoted $\G_4, \ldots, \G_{37}$.
\end{enumerate}

The irreducible real reflection groups (acting orthogonally)
are commonly designated as
$W\! A_n$, $W\! B_n$, $W\! D_n$, $W\! E_6$, $W\! E_7$, $W\! E_8$, 
$W\! F_4$, $I_r(2)$, $H_3$ and $H_4$
or some variant of this notation;
%%(the prefix ``W'' indicates Weyl group); 
see standard texts such as
Grove and Benson~\cite{GB}, Humphreys~\cite{JH} or Kane~\cite{RK}.  
We are mainly interested in groups generalizing
the infinite families $\text{Sym}_n=\G(1,1,n)$ (the symmetric group
acting by $n\times n$ permutation matrices), 
$W\! B_n=\G(2,1,n)$, and $W\! D_n=\G(2,2,n)$.
These are often called {\em permutation groups} in the literature
on group coding as they generalize the permutation group
$\G(1,1,n)$.

%%%%%%%%%%%%%%%%%%%%%%%%%%%%%%%%%%%%%%%%%%%%%%%%%%%%%%
\section{Infinite family of complex reflection groups $\G(r,1,n)$} 
\label{gr1n}

We apply the above decoding program to the 
complex reflection
groups $\G(r,1,n)$ for arbitrary integers $n,r\geq 1$
in this section.  
We obtain efficient codes with good error control 
properties that resist channel noise.
These groups are wreath products acting
by isometries on $\CC^n$, specifically, 
extensions 
of $(\mathbb Z/r\mathbb Z)^n$ by the symmetric group $\Sym_n$:
$$\G(r,1,n) \cong 
\Sym_n
\ltimes
(\ZZ/ r \ZZ )^n  
\quad\text{ and }\quad
|\G(r,1,n)|=n!\, r^n\, .
$$
Let $\om$ be the primitive complex
$r$-th root-of-unity $e^{\frac {2 \pi i}r}$, so that 
$\G(r,1,n)$ is the set all 
matrices with a single nonzero entry in each row and in each
column, that entry being a power of $\om$.

Consider the diagonal transformations
$a_i$ $(1 \leq i \leq n)$ that multiply the $i$-th entry of a vector
by $\om$ and the transpositions $b_j$ for $1 \leq j < n$ that switch the 
$j$-th and $(j+1)$-st coordinates.  Then $b_1,\ldots, b_{n-1}$
generate the symmetric group $\G(1,1,n)\leq \G(r,1,n)$
and every element of $\G(r,1,n)$ can be written uniquely as a product
of a permutation matrix (generated by the $b_i$)
and a diagonal matrix (generated by the $a_i$).
Fix an initial vector
$\x_0 = (u_1, \ldots, u_n)$ with $0<u_1<\ldots<u_n$ real.

%%%%%%%%%%%%%%%%%%%%%%%%%%%%%%%%%%%%%%%%%%%%%%%
\subsection{Defining relations for the group}
We will use the
Coxeter-like abstract presentation for $\G(r,1,n)$
in terms of generators and canonical braid relations:
$$
\begin{aligned}
\G(r,1,n)=\langle & a_1, b_1, \ldots, b_{n-1}:\ \ 
a_1^r=1=b_i^2,\\
& b_i b_j= b_j b_i\text{ for } |i-j|>1,
a_1 b_j = b_j a_1\text{ for } 1\neq j\neq 2, \\
& b_i b_{i+1} b_i = b_{i+1} b_i b_{i+1}, 
a_1 b_1 a_1 b_1 = b_1 a_1 b_1 a_1
\rangle .
\end{aligned}
$$
In other words, the following Coxeter-Dynkin diagram
gives the abstract group structure for $\G(r,1,n)$:

\hphantom{aaa}

%%%%%%%%%%%%%%% Dynikin/Coxeter diagram for G(r,1,n) %%%%%%%%%%%%%
\begin{figure}[ht]
\label{DynkinDiagram}
\begin{center} 
\begin{tikzpicture}
\node[root] (j)  {{\tiny $r$}}; 
\node[root] (k) [right=of j] {{\tiny 2}};  
\node[root] (m) [right=of k] {{\tiny 2}};  
\node[root] (n) [right=of m] {{\tiny 2}};  
\node (o) [right=of n] {$\quad\ldots\quad$}; 
\node[root] (p) [right=of o] {{\tiny 2}}; 
\node[root] (q) [right=of p] {{\tiny 2}};  
\draw[thick] (k) -- (m) -- (n); 
\draw[thick] (n) -- (o); 
\draw[thick] (o) -- (p); 
\draw[thick] (p) -- (q); 
\draw[double, double distance = 3pt] (j) -- (k); 
\end{tikzpicture} 
\end{center} 
%\caption{Dynkin-Coxeter diagram for $\G(r,1,n)$}
\end{figure}

\hphantom{aaa}

%%%%%%%%%%%%%%%%%%%%%%%%%%%%%%%%%%%%%%%%%%
\subsection{Subgroup sequence, coset leaders, and generators}
Consider the nested sequence of subgroups 
$$\{I\}=\G_0 < \G_1  < \ldots  < \G_{2n-1}=\G$$
given as block diagonal matrix groups
$$
\begin{aligned}
\G_{2\ell-1}&= \G(r,1,\ell)\oplus \{I_{n-\ell}\}
&\quad\quad\text{for }\ell=1,\ldots, n,\\
\G_{2\ell}&= \G(r,1,\ell)\oplus \G(r,1,1)\oplus\{I_{n-\ell-1}\}
&\quad\quad\text{for }\ell=1,\ldots, n-1\\
\end{aligned}
$$
where $I_\ell$ is the $\ell\times \ell$ identity matrix.
Fix coset leaders for $\G_{k}$ over $\G_{k-1}$ by setting
$$\begin{aligned}
\op{CL}(\G_{2\ell}/ \G_{2\ell -1})&=
\{I, a_{\ell+1}, a_{\ell+1}^2, \ldots, a_{\ell+1}^{r-1}\},\\
\op{CL}(\G_{2\ell+1}/ \G_{2\ell})&=
\{I, b_{\ell}, b_{\ell-1} b_{\ell}, \ldots, b_2 b_3\cdots b_{\ell},
b_1 b_2 \cdots b_{\ell} \}.
\end{aligned}
$$
We choose generators $X_k\subset \G$
for the subgroups $\G_k$ to reflect the fact that 
(at the even steps) $\G_{2k}$ is obtained by adding adding a generator 
$a_{\ell+1}$ that commutes with the elements of $\G_{2k-1}$
and (at the odd steps) 
$\G_{2\ell+1}$ is obtained by adding adding the transposition
$b_{\ell}$:
Set
$$
\begin{aligned}
X_{2\ell-1}&=\{a_1, b_1, b_2, \ldots, b_{\ell-1}\}\, ,\\
X_{2\ell}&=\{a_1, b_1, b_2, \ldots, b_{\ell-1}, a_{\ell+1}\}.\\
\end{aligned}
$$

%%%%%%%%%%%%%%%%%%%%%%%%%%%%%%%%%%%%%%%%%%%%%%%%%%%%%%%%%%%%%%5
\begin{table}[]
\caption{Subgroup sequence for $\G(r,1,n)$} \label{table:first}
\centering
\begin{tabular}{|c|c|c|}
\hline
{$k$}&{Generating set $X_k$ for $\G_k$}&{Coset leaders for $\G_k$ over $\G_{k-1}$}
\rule[-1ex]{0ex}{3.5ex}%strut
\\
\hline
$0$ & $I$ & {} 
\rule[-1ex]{0ex}{3.5ex}%strut
\\
\hline
$1$ & $a_1$ & $a_1, a_1^2, \ldots, a_1^r=I$
\rule[-1ex]{0ex}{3.5ex}%strut
\\
\hline
$2$ & $a_1, a_2$ & $a_2, a_2^2, \ldots, a_2^r=I$
\rule[-1ex]{0ex}{3.5ex}%strut
\\
\hline
$3$ & $a_1, b_1$ & $I, b_1$ 
\rule[-1ex]{0ex}{3.5ex}%strut
\\
\hline
$4$ & $a_1, b_1, a_3$ & 
$a_3, a_3^2, \ldots, a_3^r=I$
\rule[-1ex]{0ex}{3.5ex}%strut
\\
\hline
$5$ & $a_1, b_1, b_2$ & $I, b_2, b_1b_2$ 
\rule[-1ex]{0ex}{3.5ex}%strut
\\
\hline
$\vdots$ & $\vdots$ & $\vdots$
\rule[-1ex]{0ex}{3.5ex}%strut
 \\
\hline
$2n-2$ & $a_1, b_1, \ldots, b_{n-2}, a_n$ & $a_n, a_n^2, \ldots, a_n^{r}=I$
\rule[-1ex]{0ex}{3.5ex}%strut
\\
\hline
$2n-1$ & $a_1, b_1, \ldots, b_{n-1}$ & $I, (b_j\cdots b_{n-1})$ 
for $1 \leq j \leq n-1$
\rule[-1ex]{0ex}{3.5ex}%strut
\\
\hline
\end{tabular}
\end{table}
\
%%%%%%%%%%%%%%%%%%%%%%%%%%%%%%%%%%%%%%%%%%%%%%%%%%%%%%%%%%%%%%%5
\subsection{Correct and Robust Decoding}

The above choices coincide with 
the natural choice of subgroup sequence, coset leaders, and initial vector for general wreath products given in Section~\ref{wreath}.
Thus Corollary~\ref{wreathdecodescorrectly} implies
%%%%%%%%%%%%%%%%5
\begin{cor}\label{rightchoice}
With the above choice of subgroup sequence, coset leaders, and initial vector,
the subgroup decoding algorithm for $\G(r,1,n)$ (for any $r$ and any $n$)
decodes robustly:
For all $g$ in $\G(r,1,n)$, any received vector in the decoding region
of $g$ decodes to $g$.
\end{cor}
%%%%%%%%%%%%%%%

\hphantom{aaa}

\vspace{0ex}

\hphantom{aaa}

%%%%%%%%%%%%%%%%%%%%%%%%%%%%%%%%%%%%%%%%%%%%%%%%%%%%%%%%%%%%%
\subsection{Implementing the Decoding Algorithm Explicitly}

Although the last corollary shows that the algorithm decodes correctly,
it is helpful to point out explicitly how one implements
the algorithm by hand 
using the ideas in the proof of Theorem~\ref{wreathgreedy}.
Suppose $\r=(x_1, \ldots, x_n)$ is a received vector in $\CC^n$,
and recall that $\x_0 = (u_1, \ldots, u_n)$ with $0<u_1<\ldots<u_n$ real.
Consider the sequence
\begin{align*}
 \Vert \r &- \x_0 \Vert \\
 \Vert a_1^k\r &- \x_0 \Vert \\
 \Vert a_2^{\ell}a_1^k\r &- \x_0 \Vert \\
 \Vert b_1^{\delta}a_2^{\ell}a_1^k\r &- \x_0 \Vert \\
 \Vert a_3^m b_1^{\delta}a_2^{\ell}a_1^k\r &- \x_0 \Vert \\
 \Vert c a_3^m b_1^{\delta}a_2^{\ell}a_1^k\r &- \x_0 \Vert \\
{} \vdots &{}
\end{align*}
where $c$ is a coset leader for $\G_5$ over $\G_4$, thus one of
$\{ I, b_2, b_1b_2 \}$.
First $k$ is chosen to maximize $\op{Re}(\om^k x_1)$, 
then $\ell$ to maximize $\op{Re}(\om^{\ell} x_2)$. 
Now since $u_1 < u_2$, an easy calculation shows that if
$\op{Re}(\om^k x_1) > \op{Re}(\om^{\ell} x_2)$, then we should 
apply $b_1$, switching the values, to minimize the distance; 
otherwise not (so that $\delta$ is $0$ or $1$).  
Next $m$ is chosen to maximize $\op{Re}(\om^m x_3)$. 
Then, since $u_1 < u_2 < u_3$, we apply the correct coset leader $c$ 
(a permutation) to put
$\op{Re}(\om^k x_1)$, $\op{Re}(\om^{\ell} x_2)$, $\op{Re}(\om^m x_3)$
into increasing order (an insertion sort).  Continue until pau.

\hphantom{aaa}

\begin{remark}\label{speedup}{\em
An observation in the proof of Theorem~\ref{wreathgreedy}
can be used to speed up the algorithm 
considerably.  Writing $x=|x|e^{i\theta}$, 
we maximize the real part of
$\om^kx = |x|e^{(\frac{2 \pi k}r + \theta)i}$
by making 
$\frac{2 \pi k}r + \theta$ as close to $2\pi$ as possible.
Thus $k$ should be chosen as the nearest integer to $r-\frac{r\theta}{2\pi}$.}
\end{remark}
%%%%%%%%%%%%%%%%%%%%%%%%%%%%%%%%%%%%%%%%%%%%%%%

%%%%%%%%%%%%%%%%%%%%%%%%%%%%%%%%%%%%%%55
\subsection{Initial vector}
We refine our choice of initial vector
so that neighbors of $\x_0$ are just its images under
the natural generating set $a_1, b_1, \ldots, b_{n-1}$
in order to control errors.
We mimic construction of an 
optimal vector for the Coxeter group $W\! B_n$.
If we take a real vector $\x_0$ of the form
\[   \x_0 = (  \alpha, \alpha+\beta, \alpha+2\beta, \ldots ,\alpha +(n-1)\beta )  \]
and require that 
\[ \| a_1 \x_0 - \x_0 \| = \| b_1 \x_0 - \x_0 \| = \ldots =
\| b_{n-1} \x_0 - \x_0 \| , \]
then a straightforward computation gives
\[   \frac \beta \alpha = \sqrt{1 - \cos \frac{2\pi}r}  \]
with $\sqrt{2} \beta$
as the minimum distance of the code. 
Initially we set $\alpha=1$, and then normalize so that $\| \x_0 \|=1$.
Note that $\| a_i\x_0 - \x_0 \|$ will be greater than $\sqrt{2} \beta$
for $i>1$.
This choice gives an initial vector with full orbit under $\G$, 
and the minimum distances of the code
defined by this choice of $\x_0$ (for various $r$ and $n$)
have a reasonable order
of magnitude.  
Table~\ref{table:dmingot} gives the values achieved for small values of
$r$ and $n$.

%%%%%%%%%%%%%%%%%%%%%%%%%%%%%%%%%%%%%%%%%%%%%%%%%%%%%%%%%%%%%5
\begin{table}[h]
\caption{Actual $d_{min}$ obtained for some $\G(r,1,n)$} 
\label{table:dmingot}
\centering
\begin{tabular}{|c|c|c|c|}
\hline
{$r$}& $n=2$ & $n=3$ & $n=4$ \rule[-1ex]{0ex}{3.5ex}%strut
\\
\hline
$3$ & .71  & .41 & .27  \rule[-1ex]{0ex}{3.5ex}%strut
\\
\hline
$4$ & .63  & .38 & .26  \rule[-1ex]{0ex}{3.5ex}%strut
\\
\hline
$5$ & .56  & .35 & .24  \rule[-1ex]{0ex}{3.5ex}%strut
\\
\hline
$6$ & .51  & .32 & .23  \rule[-1ex]{0ex}{3.5ex}%strut
\\
\hline
$7$ & .46  & .30  & .21  \rule[-1ex]{0ex}{3.5ex}%strut
\\
\hline
$8$ & .42  & .28 & .20  \rule[-1ex]{0ex}{3.5ex}%strut
\\
\hline
\end{tabular}
\end{table}

%%%%%%%%%%%%%%%%%%%%%%%%%%%%%%%%%%%%%%%%%%%%%%%%%
\begin{remark}{\em
Notice that the fundamental regions depend on the choice of the initial 
vector $\x_0$ unlike the case of group coding
over the real numbers.  For example, consider 
the first subgroup in the subgroup sequence,
$\G_1=\la a_1 \ra$.  Writing $\x = (x_1, \ldots, x_n)$ 
and $\x_0 = (u_1, \ldots, u_n)$, we have
\begin{align*}
 \op{FR}(\G_1) &= \{ \x \in \mathbb C^n : 
  \| a_1^k\x-\x_0 \| > \| \x - \x_0 \| \text{ for } 1 \leq k < r  \} \\
   &= \{ \x \in \mathbb C^n : \op{Re}(x_1 \overline{u_1}) >
   \op{Re}(\om^k x_1 \overline{u_1}) \text{ for } 1 \leq k < r \} .
\end{align*}
This justifies in part our standard choice of $\x_0$ as indicated above.
}
\end{remark}

%%%%%%%%%%%%%%%%%%%%%%%%%%%%%%%%%%%%%
\subsection{Controlling Errors}
The above choices of subgroup sequence, coset leaders, and initial vector
for $\G(r,1,n)$ are consistent with those from Section~\ref{wreath}
for general wreath products.  Thus Proposition~\ref{errorcontrolwreath}
implies Error Control Property~\ref{propWes} for $\G(r,1,n)$.

We now check directly that 
Nearest Neighbors Property~\ref{naygen} holds
as well, i.e., we check
that if $g\x_0$ is any nearest neighbor of $\x_0$, then
$g$ lies in $X_{\G}\cup X_{\G}^{-1}$.
We argue that if
$g \ne I, a_1, a_1^{-1}$ or some $b_j$, then 
$\Vert g\x_0 - \x_0 \Vert^2 > 2\beta^2=d_{min}$.
This distance squared is the sum of (at least two)
terms of the form
\[  
| \om^t (\alpha+j\beta) - (\alpha+\ell \beta) |^2 
\]
for integers $j$ and $\ell$.
One may verify that if $j=\ell$ and $t \ne 0$,
then this expression is at least $2\beta^2$, while if $j \ne \ell$, then it is 
at least $\beta^2$.  

%If $j = \ell$ then 
%\begin{align*}
%| \om^t (\alpha+j\beta) - (\alpha+j \beta) |^2 &= |(\om^t -1) \alpha|^2 \\
%   &= (2 - 2 \op{Re}(\om^t)) \alpha^2 \\
%   &\geq (2 - 2 \op{Re}(\om)) \alpha^2 \\
%   &= 2 \beta^2 \ ,
%\end{align*}
%as desired.  If $j \ne \ell$, then by the triangle inequality
%\begin{align*}
%| \om^t (\alpha+j\beta) - (\alpha+\ell \beta) | &\geq
%  \left| | \om^t (\alpha+j\beta)| - |(\alpha+\ell \beta) | \right|  \\
%  &= \left| | (\alpha+j\beta)| - |(\alpha+\ell \beta) | \right|  \\
%  &\geq \beta
%\end{align*}
%which gives at least $\beta^2$ for the square.

Corollary~\ref{controlnoise} then implies error control for the groups
$\G(r,1,n)$:
%%%%%%%%%%%%%%%%%%%%%%%%%%%%%%%%%%%%%%%%%%%%%%
\begin{cor}
For $\G=\G(r,1,n)$, 
assume the above natural choice for subgroup sequence, initial vector,
coset leaders, and 
generating sets $X_k\subset \G$ for each $\G_k$.
If a received vector lands in the decoding region 
containing a nearest neighbor of $g^{-1}\x_0$
due to noise, then the subgroup decoding algorithm decodes it to a group element
differing from $g$ in only one factor when written
as a product of coset leaders.
\end{cor}
%%%%%%%%%%%%%%%%%%%%%%%%%%%%%%%%%%%%%

\vspace{1ex}

%%%%%%%%%%%%%%%%%%%%%%%%%%%%%%5
\subsection{Efficient Decoding: Navigating Coset Leader Graphs}

We argue that the above choices for $\G(r,1,n)$
also yield efficient decoding using navigation through the 
coset leader graphs as described in Section~\ref{type3}. 
One can check directly that each coset leader graph
is connected (see Definition~\ref{propCL}) for $\G(r,1,n)$.
The graphs for $\G(4,1,4)$ are given in
Figure~1.
%\ref{fig:clbbb5}.
(Note that Kriloff and Lay~\cite{KriloffLay} show existence of
Hamiltonian cycles for the Cayley graphs of $\G(r,1,n)$.)
We use 
Remark~\ref{speedup} and the explicit decoding process 
described after Corollary~\ref{rightchoice}. 
At stages $1,2,4, \ldots, 2k$ 
where the coset leader graphs are cyclic,
we can choose in one step the coset leader that moves the received
vector closest to the initial vector.
For the permutation stages $3,5,\ldots,2k+1$ the graph gives an insertion
sort.  As in~\cite{FNP}, a modified insertion sort could also be used
to shorten the decoding somewhat.  Hence the coset leader graphs for
$\G(r,1,n)$ are particularly easy to navigate, compared to most unitary
groups.

% The counter doesn't match the label, so hacking:
\makeatletter 
\renewcommand{\thefigure}{1}%S\@arabic\c@figure}
\makeatother
\begin{figure}[tbp]\label{cosetleadergraphs}.
%\label{fig:clbbb5}
\begin{center}
\includegraphics[height=2in,width=6in]{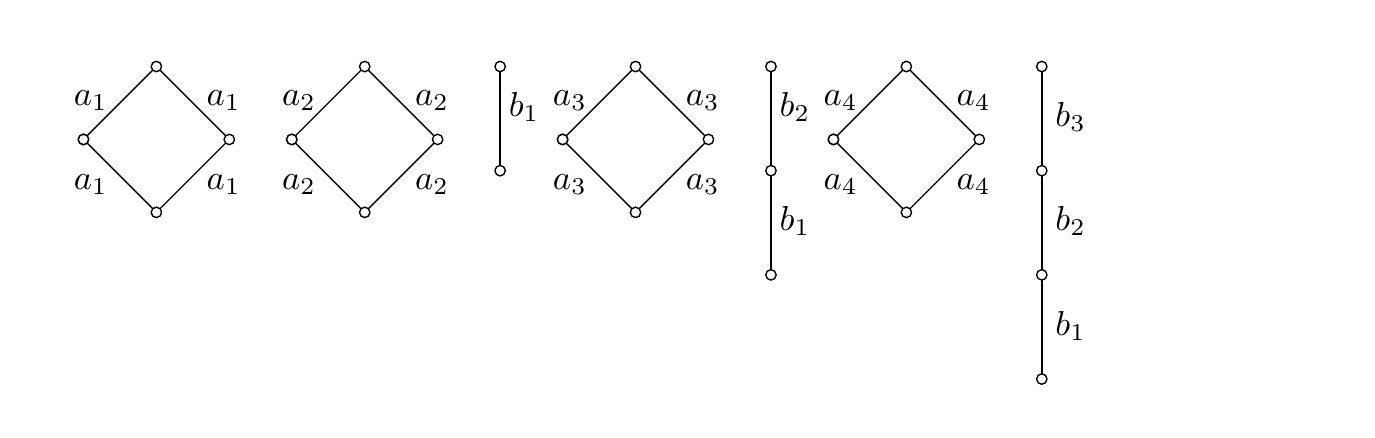}
\caption{Coset leader graphs for $\G(4,1,4)$}%
\end{center}
\end{figure}

\vspace{0ex}

%%%%%%%%%%%%%%%%%%%%%%%%%%%%%%%%%%%%%%%%%%%%%%%%%%%%%%%5
\subsection{Efficient Decoding: Number of Steps in the Algorithm}
Assuming that we use the method indicated in the last
subsection to navigate the cyclic
coset leader graphs, 
the analysis of the average number of steps to decode using $\G(r,1,n)$
is identical to that given for the Weyl group $W\! B_n=\G(2,1,n)$ in Fossorier, Nation and 
Peterson~\cite{FNP}.  In other words, for any $r \geq 2$, one can decode 
$\G(r,1,n)$ just as fast as $\G(2,1,n)$.  
Moreover, exactly as in~\cite{FNP}, one can speed up the sorting by using
a slightly different subgroup sequence, which amounts to using an improved
insertion sort.  We omit the details
and give the results.  

Asymptotically, the
number of steps in decoding is $\frac {n^2}4$ for the
subgroup sequence given here, and $\frac {n^2}8$ for the modified sort.
But for moderate values of $n$, the number of steps is fewer than that
would indicate, and in fact close to the theoretical minimum.
Some of these numbers are given in Table~\ref{tab:Bncomp}, where
\begin{itemize}
\item $\gamma_n\,$ is the average number of comparisons 
  to decode using intermediate subgroups with a standard insertion sort,
\item $\gamma_n'\,$ is the average number of comparisons 
  to decode using intermediate subgroups with a modified insertion sort,
\item $n+\log_2 n!\,$ is the theoretical minimum average number of
  comparisons; see Knuth~\cite{Knuth}.
\end{itemize}

%%%%%%%%%%%%%%%%%%%%%%%%%%%%%%%%%%%%%%%%%
\begin{table}[h]
\caption{Average number of comparisons to decode $\G(r,1,n)$} 
\label{tab:Bncomp}
\centering
\begin{tabular}{|c|c|c|c|}
\hline
{$n$}&{$\gamma_n$}&{$\gamma'_n$}&{$n+\log_2 n!$}
\rule[-1.5ex]{0ex}{4ex}%strut
\\
\hline
4 & 8.9 & 8.7 & 8.6 \rule[-1ex]{0ex}{3.5ex}%strut
\\
\hline
8 & 27.3 & 24.0 & 23.3 \rule[-1ex]{0ex}{3.5ex}%strut
\\
\hline
16 & 88.6 & 67.7 & 60.3 \rule[-1ex]{0ex}{3.5ex}%strut
\\
\hline
32 & 307.9 & 204.5 & 149.7 \rule[-1ex]{0ex}{3.5ex}%strut
\\
\hline
\end{tabular}
\end{table}
%%%%%%%%%%%%%%%%%%%%%%%%%%%%%%%%%%%%%%%%%%%%%%%%%%%%5
%%%%%%%%%%%%%%%%%%%%%%%%%%%%%%%%%%%%%%%%%%%%%%%%%%%%5
\subsection{Quaternions}
There is an obvious generalization of the groups $\G(r,1,n)$ that will
have the same good decoding properties.  These are the groups 
$\mathbb{P}(\K,n)$ of all $n \times n$ permutation matrices whose 
nonzero entries are from a group $\K$ of complex numbers $z$ with $|z|=1$,
or more generally, quaternions $w$ with $|w|=1$.  For example, we could 
take 
$$\K = \{ z \in \mathbb C : z^{2^k} = 1 \text{ for some } k \geq 1 \}\ .
$$
This is an infinite group, but for any given application we would only use
a finite part of it, although without a predetermined bound.  
Likewise, there are a few finite multiplicative subgroups of 
unit quaternions that could be used as entries in the permutation matrices;
see Kranek~\cite{WK} or Lehrer and Taylor~\cite{LT}.
As an exercise, we programmed a simulation of coding with $\mathbb P(\H,3)$
with $\H$ the 8-element quaternion group.

\vspace{2ex}

\hphantom{aaa}

%%%%%%%%%%%%%%%%%%%%%%%%%%%%%%%%%%%%%%%%%%%%%%%%%%%%%%%%%%%%%%%5
%%%%%%%%%%%%%%%%%%%%%%%%%%%%%%%%%%%%%%%%%%%%%%%%%%%%%%%%%%%%%%%%%%5
%%%%%%%%%%%%%%%%%%%%%%%%%%%%%%%%%%%%%%%%%%%%%%%%%%%%%%%%%%%%%%%%%%%
\section{Other complex reflection groups}
\label{othergroups}

\subsection{Subgroups of $\G(r,1,n)$}
For any divisor $p$ of $r$, recall that $\G(r,p,n)$ is a reflection
subgroup of $\G(r,1,n)$.
The properties that make subgroup decoding work well
for the groups $\G(r,1,n)$ seem not to hold for the groups
$\G(r,p,n)$ with $p>1$,
except
for the real group $W\! D_n=\G(2,2,n)$ (see~\cite{FNP}).
A general choice of
subgroup sequence, initial vector, and coset leaders that is greed
compatible seems
elusive.  In addition, we have not been able to find choices
giving the Error Control Property~\ref{errorcontrol}.
This leaves the question:  
\emph{Is there any good decoding scheme for the groups $\G(r,p,n)$
with $p>1$?}

%%%%%%%%%%%%%%%%%%%%%%%%%%%%%%%%%%%%%%%%%%%%%%%%%%%%%%%%%%%%%%%%

\subsection{Tetrahedral group $\G_4$, Octahedral $\G_8$, Icosehedral $\G_{16}$}
\label{G_8}

In Section~\ref{G_4} we saw that
subgroup decoding worked for codes based
on the tetrahedral group $\G_4$ using a careful choice of the subgroup
sequence and initial vector.
 There are two other reflection groups
of this type, the octahedral group $\G_8$ and the icosahedral group $\G_{16}$.
These groups are generated by matrices $A$ and $B$ satisfying the equations
$A^k=B^k=I$ and $ABA=BAB$ for $k=3$, 4 and 5 respectively:
\begin{itemize}
   \item $k=3$ gives $\G_4$ with 24 elements.
   \item $k=4$ gives $\G_8$ with 96 elements.
   \item $k=5$ gives $\G_{16}$ with 600 elements.
\end{itemize}

For the octahedral group, if we take the natural subgroup sequence
$\{ I \} < \{ I,A,A^2,A^3 \} < \G_8$ and a \emph{nonreal} unit vector $\x_0$
such that $\| A^{-1}\x_0 - \x_0 \| = \| B^{-1}\x_0 - \x_0 \|$, 
then coset leaders can be chosen minimal
and the subgroup decoding 
algorithm decodes correctly with some noise.

On the other hand, we have not been able to find a combination of subgroup
sequence and initial vector that gives minimal coset leaders for a code
based on the icosahedral group $\G_{16}$.  For example, 
for a standard matrix representation
and subgroup sequence $\{ I \} < \{ I,A,A^2,A^3,A^4 \} < \G_{16}$,
ties arise in a rather unexpected way:  
\[
B^3A^4B^3 = \begin{bmatrix} c & 0 \\ 0 & c \end{bmatrix}
\qquad\text{and}\qquad
B^3A^4B^3A^4 = \begin{bmatrix} 
 c & 0 \\ 0 & \overline{c} 
\end{bmatrix}
\]
where $c = e^{\frac {\pi}5 i}$.  

%%%%%%%%%%%%%%%%%%%%%%%%%%%%%%%%%%%%%%%%%%%%%%%%%%%%5
\subsection{Hessian groups $\G_{25}$ and $\G_{26}$}  
On the other hand, the complex reflection groups $\G_{25}$ and $\G_{26}$
do not admit any subgroup decoding scheme as far as we can tell.
Despite repeated attempts, using computerized search programs,
we have been unable to find a subgroup sequence
and initial vector such that subgroup decoding works
for these groups.

%%%%%%%%%%%%%%%%%%%%%%%%%%%%%%%%%%%%%%%%%%%%%%%%%%%%%%%%%%
%%%%%%%%%%%%%%%%%%%%%%%%%%%%%%%%%%%%%%%%%%%%%%%%%%%%%%%%%%
%%%%%%%%%%%%%%%%%%%%%%%%%%%%%%%%%%%%%%%%%%%%%%%%%%%%%%%%%%
\section{Conclusions} \label{conk}

Subgroup decoding works well for codes based on the groups $\G(r,1,n)$,
which are wreath products of cyclic groups, thus generalizing codes
based on the real reflection groups $W\! A_n \cong \text{Sym}_n$ 
and $W\! B_n$.
Codes based on these groups decode robustly, have good error control,
and decode in few steps relative to the size of the group.
There are problems with error control (Property~\ref{propWes}) for the 
groups $\G(r,p,n)$ with $p>1$ that generalize $W\! D_n$.
Subgroup decoding works on some of the exceptional unitary groups, 
but not others, and this seems to be inherent in the structure of the groups.
In general, good coding properties are preserved by wreath products,
allowing us to build large codes from small ones.

This suggests that other decoding methods should be considered.  
Walker, building on the work of Kim~\cite{HJK}, has designed an 
alternative algorithm for arbitrary unitary groups
called the {\em Snowflake Algorithm}; see~\cite{CW, JBNCW}.
The efficiency of this other decoding method varies from pretty good to very
good, depending on the group action, in ways that we do not yet totally
understand.  In the Snowflake algorithm,
the basic algorithm of group coding, transmitting $g^{-1}\x_0$ and
decoding with $g'(\r) \approx \x_0$, remains unchanged.
However, the use of a subgroup sequence is abandoned, so that the greedy 
aspect of the algorithm is no longer a factor. 
Rather, a set of generators is chosen for $\G$
so that each group element will have a relatively short expression
as a product of the generators.  This expression
may not be unique, but one such
expression can be chosen as a canonical form for the element and tables
of equivalent minimal expressions calculated.
Using these, one can decode \emph{correctly} with some noise, and for some groups
it can be done \emph{efficiently}.  For those groups where the
 algorithm can be made efficient, including wreath products of the complex 
reflection groups $\G_4$, $\G_5$, $\G_8$ and $\G_{20}$,
the Snowflake algorithm might provide an alternative method of decoding group codes.

\hphantom{aaa}

%%%%%%%%%%%%%%%%%%%%%%%%%%%%%%%%%%%%%%%%%%%%%%%%%%%%%%%%%%%%%%%%%%%%
%%%%%%%%%%%%%%%%%%%%%%%%%%%%%%%%%%%%%%%%%%%%%%%%%%%%%%%%
\section{Appendix I:  A primitive group decoding algorithm}

This paper has focused on subgroup decoding, which works very well for 
codes based on real reflection groups or the groups $\G(r,1,n)$.  
These group codes may prove useful in certain 
practical situations.  The same probably cannot be said for codes
based on arbitrary unitary groups, though there may be applications which
we cannot yet envision, e.g., in cryptography.   
Often, a choice of initial vector and subgroup sequence yielding
an effective decoding algorithm (or one that even decodes correctly)
remains elusive.   
In this appendix, we describe a very general type of decoding algorithm. 
Then we give an analog of Theorem~\ref{shrinkregion}:  If a weak necessary
condition is satisfied, then the algorithm decodes correctly when the received
vector is sufficiently close to the sent codeword.
The appendix is based on Kim~\cite{HJK}; a refined version is given in
Walker~\cite{CW}.

The parameters for this type of decoding are
\begin{itemize}
\item a finite unitary group $\G$,
\item an initial unit vector $\x_0$,
\item a generating set $X$ for $\G$.
\end{itemize}
Again, the codewords are elements of the orbit $\G\x_0$, a codeword 
$\x=g^{-1}\x_0$ is transmitted, and the received vector is $\r=\x+\n$
where $\n$ represents noise.
The {\em primitive decoding algorithm} 
decodes as follows.  We fix some predetermined $\varepsilon >0$.
Let $\r_0=\r$. Recursively, given $\r_k$, find a transformation 
$c_{k+1} \in X$ such that the vector $\r_{k+1} = c_{k+1}\r_k$
satisfies 
$$\| \r_{k+1} - \x_0 \| < \| \r_k - \x_0 \| - \varepsilon\ .$$  
If no such $c_{k+1}$ exists, terminate and 
decode $\r$ as $c_k \cdots c_1$.

For example, if $\G$ is a reflection group, we might take $X$ to be all reflections
or a minimal generating set of reflections or anything in between.
(Walker has shown that it may be necessary to include
some nonreflections in the set $X$ to obtain the condition $(\ddagger)$ 
below.)

Let us assume that the pair $X$, $\x_0$ 
satisfies the condition that every
nontrivial codeword is sent closer to the initial vector
by some element of $X$:
\[ (\ddagger )\ \  \text{For any } \w \in \G\x_0 \text{ with }
\w\neq \x_0, 
\Vert  c\w-\x_0 \Vert < \Vert  \w-\x_0 \Vert
\text{ for some }
c \in X .\]
(This is a condition satisfied by simple reflections in a Coxeter group:
Every group element factors as a product of a minimum number
of simple reflections generating the group, 
multiplying by the first factor decreases
length, and length corresponds to distance back to 
some initial vector.)

We want to show that the procedure terminates and decodes correctly,
i.e., at termination $c_k \cdots c_1 \in  \S g$ where $\S = \op{Stab}(\x_0)$. 
Clearly $(\ddagger)$ is necessary for correct decoding, for if $\w$ witnesses
a failure of $(\ddagger)$, then $\w$ cannot be decoded correctly even 
with no noise.
For each codeword $\w$, let $\op{MG}(\w)$ be the set of 
``minimal generators'' $c$  
that minimize the distance from $c\w$
back to $\x_0$ over all $c$
in $X \cup \{I \}$:
\[
\op{MG}(\w) = \big\{ c \in X \cup \{ I \} : 
\Vert  c\w-\x_0 \Vert  
\leq
\Vert d\w-\x_0 \Vert
\ \text{ for all } d\in X\cup\{I\}
\big\}.
\]
Then $(\ddagger )$ 
is equivalent to the condition that
$I \notin \op{MG}(\w)$
whenever codeword $\w \ne \x_0$.
Define
\[ \delta = \min_{
\substack{
\w \in \G\x_0 - \{ \x_0 \}
\rule{0ex}{2ex}%strut
 \\ c\, \in \op{MG}(\w)
\rule{0ex}{2ex}%strut
}}
            \Vert \w -\x_0 \Vert - \Vert c\w -\x_0 \Vert  \]
so that $\Vert \w -\x_0 \Vert \geq \Vert c\w -\x_0 \Vert + \delta$
for any $c\in \op{MG}(\w)$. 
There are two versions of the algorithm.  At each step, either
\begin{enumerate}
\item[(A)]  choose $c_{k+1}$ to minimize $\Vert c_{k+1}\r_k -\x_0 \Vert$, or 
\item[(B)]  choose the first $c_{k+1}$ such that 
$\Vert c_{k+1}\r_k -\x_0 \Vert < \Vert \r_k -\x_0 \Vert - \frac 13 \delta$. 
\end{enumerate}
In either version, when there is no $c \in X$ such that
$\Vert c\r_k -\x_0 \Vert < \Vert \r_k -\x_0 \Vert - \frac 13 \delta$,
we terminate and decode $\r$ as $c_k \cdots c_1$.

%The threshold $\frac 12 \delta$ in (B) can be replaced by any $\theta$ with
%$0 < \theta < \delta$; the proof below essentially uses $\theta = (1-2f)\delta$
%where $0<f<\frac 12$.

We verify that either version of the
primitive decoding algorithm works with some noise:

%%%%%%%%%%%%%%%%%%%%%%%%%%%%%%%%%%%%%%%%%%%%
\begin{thm} \label{smallnoisez}
Assume that the pair $X$, $\x_0$ satisfies the condition $(\ddagger)$.
Define $\delta$ as above.
If\/ $\Vert \r - g^{-1}\x_0 \Vert < \delta/ 3$, then the procedure
terminates in at most $\lfloor 6/\delta \rfloor$ steps 
and outputs $c_k \cdots c_1
\in g\S$.
\end{thm}
%%%%%%%%%%%%%%%%%%%%%%%%%%%%%%%%%%%%%%5

\begin{proof}
We show that each step of the algorithm moves us at least $\delta/ 3$
closer to the initial vector.
Hence the process terminates in at most 
\[ (3 / \delta)\ 
\max \Vert \w - \x_0 \Vert 
\leq (3/ \delta) \,  2
=6 / \delta \]
steps (not counting a possible terminal step of choosing $I$), where
the max is taken over all codewords $\w$ (on the unit sphere).

At step $k$, set $g' = c_k \cdots c_1$, $\w = g'g^{-1}\x_0$, and
$\r_k = g'\r$.  Suppose $\w \ne \x_0$.  Note that
$\Vert \r_k - \w \Vert = \Vert g'\r-g'g^{-1}\x_0 \Vert < \delta/ 3$.  
By $(\ddagger)$ and the definition of $\delta$, 
there exists $c \in \op{MG}(\w)$ with
\begin{align*}
 \Vert  \r_k-\x_0 \Vert &\geq \Vert \w-\x_0 \Vert - \Vert \r_k-\w \Vert  \\ 
 &>   \Vert c\w-\x_0 \Vert + \delta - \delta/ 3  \\ 
 &=    \Vert c\w-\x_0 \Vert + 2\delta/ 3  
\end{align*}
whilst 
\begin{align*}
 \Vert  c\r_k-\x_0 \Vert &\leq  
\Vert c\w-\x_0 \Vert + \Vert c\r_k-c\w \Vert \\ 
 &<  \Vert c\w-\x_0 \Vert + \delta/ 3\ . 
\end{align*}
Thus
\[ 
\Vert  c\r_k-\x_0 \Vert 
< \Vert  \r_k-\x_0 \Vert -  \delta/ 3 \]
making $c\r_k$ closer than $\r_k$ to $\x_0$ by a step of length 
at least $\delta/ 3$ as desired.
\end{proof}

%%%%%%%%%%%%%%%%%%%%%5
%\begin{cor}
%Assume initial vector $\x_0$ has full orbit.
%Then every element of $\G$ can be written as a product of at most
%$\frac 2{\delta}$ elements of $X$.
%\end{cor}
%%%%%%%%%%%%%%%%%%%%%%%%%%%%%%%%%%%%%%%%%%%%%%%%%%%%%%%%%%%%%%

%%%%%%%%%%%%%%%%%%%%%%%%%%%%%%%%%%%%%%%%%%%%%%%%%%%%%%%%%%%%%%%%%%%
\section{Appendix II:  Partial group codes based on $\G(r,1,n)$}

It can be advantageous to use a group code based on a proper subset
of the codewords, $\W \subset \G\x_0 = \{ g\x_0 : g \in \G \}$.
In this appendix, we briefly indicate how this can be done to yield a
significant improvement in codes based on $\G(r,1,n)$.

Although the code based on $\G(r,1,n)$ in Section~\ref{gr1n}
has good error control
properties, a problem arises:
the distance between adjacent codewords is not uniform,
which makes the decoded ``bits" not uniformly reliable.
(Errors are more likely in the parts of the received vector
corresponding to smaller components of the initial vector.)
This stems from the fact that the initial
vector, 
\[  \x_0 = (  \alpha, \alpha+\beta, \alpha+2\beta, \ldots,
    \alpha+(n-1)\beta ), \]
gives $d_{min}=\sqrt{2}\beta$ 
as the minimal distance of the code 
where $0 < \beta < \alpha$ and $\beta / \alpha =
(1 - \cos \frac{2\pi}r)^{1/2}$.
 For the generators $a_i$ and $b_j$ of $\G(r,1,n)$, this choice
implies that
\[ \| a_1 \x_0 - \x_0 \| = \| b_1 \x_0 - \x_0 \| = \ldots =
\| b_{n-1} \x_0 - \x_0 \| = d_{min}   \]
and $\| a_j \x_0 - \x_0 \| > \sqrt{2}\beta$ for $j>1$. 

One solution to this problem is the following.  Recall that any group
element $g \in \G(r,1,n)$ can be written as a product of coset
leaders in the form
$$g = \tau_{\ell_n} a_n^{k_n} \cdots \tau_{\ell_3} a_3^{k_3} 
   \tau_{\ell_2} a_2^{k_2} a_1^{k_1}$$                              
where each $\tau_{\ell_j}$ is a permutation and each $k_i \in \mathbb{N}$.
Choose integers $m_j$ for $1 \leq j \leq n$ with 
$m_j$ dividing $m_{j+1}$,
  \[ 1=m_n \,|\, m_{n-1} \,|\, \dots \,|\, m_2 \,|\, m_1 \,|\, r \, .\]
%% jbn based on the calculations below.
Then use only codewords $g\x_0$ (as above) with
$m_j \,|\, k_j$ for $1 \leq j \le n$.
Although this code is a proper subset of the full code 
for $\G(r,1,n)$, it does not correspond
to a subgroup.  Note that the size of the code is
\[ |W| = \frac {n!\, r^n}{\prod_{1 \leq j \leq n} m_j}\ \ . \]
The decoding algorithm is unchanged, except that the received vector
is interpreted to be the nearest \emph{codeword}.

Now the object is to adjust the parameters $m_1, \ldots, m_{n-1}$ and
the initial vector $\x_0$ to
make as uniform as possible the distances 
        $\| b_j \x_0 - \x_0 \|$ for $1 \leq j \leq n-1$, and   
        $\| a_k^{m_k} \x_0 - \x_0 \|$ for $1 \leq k \leq n$, 
while increasing the minimum distance of the code in the process.
In practice this can be done rather effectively by \emph{ad hoc} adjustments,
but an interesting problem arises: \emph{Find a good algorithm to adjust these
parameters.}

For example, consider the code based on $\G(16,1,4)$.  The original
subgroup decoding scheme takes
$m_1=m_2=m_3=m_4=1$ and an initial vector of the form
\[  \x_0 = (  \alpha, \alpha+\beta, \alpha+2\beta, \ldots,
    \alpha+(n-1)\beta ) \]
with $\beta/\alpha = .2759$.
The size of the code is $16^4 \cdot 4! = 2^{16} \cdot 24$.
One can calculate that the variation in the
distances $\| g\x_0 - \x_0 \|$ with $g \in \{ a_1,a_2,a_3,a_4,b_1,b_2,b_3 \}$
is max/min = 1.83, and the normalized $d_{min}$ is .169.

If instead we take $m_1=4$, $m_2=2$, $m_3=m_4=1$ and $\beta/\alpha = 1.0$,
then we obtain a code with only $2^{13} \cdot 24$ codewords.  
However, the variation
in the distances is then max/min = 1.36, and the minimum distance $d_{min}$ becomes .280,
giving a considerable improvement.

%%%%%%%%%%%%%%%%%%%%%%%%%%%%%%%%%%%%%%%%%%%%%%%%%%%%%%%%%%%%%%%%%%%%%%%%%%%%%%%%%
%%%%%%%%%%%%%%%%%%%%%%%%%%%%%%%%%%%%%%%%%%%%%%%%%%%%%%%%%%%%%%%%%%%%%%%%%%%%%%%%%
%%%%%%%%%%%%%%%%%%%%%%%%%%%%%%%%%%%%%%%%%%%%%%%%%%%%%%%%%%%%%%%%%%%%%%%%%%%%%%%%%
%%%%%%%%%%%%%%%%%%%%%%%%%%%%%%%%%%%%%%%%%%%%%%%%%%%%%%%%%%%%%%%%%%%%%%%%%%%%%%%%%

\end{document}